\documentclass{article}
\usepackage{amssymb,amsmath}
\usepackage[latin1]{inputenc}
\usepackage[T1]{fontenc}
\usepackage[french,english]{babel}
\usepackage{amsthm}
\usepackage{amsfonts}
\usepackage{pxfonts}
\usepackage{dsfont}
\usepackage{enumerate}
\usepackage{indentfirst}
\author{Dominique MALICET}
\date{}
\title{Random walks on $\mathrm{Homeo}(S^1)$}
\newtheorem{thm}{Theorem}

\newtheorem{lem}{Lemma}[section]
\newtheorem{cor}[lem]{Corollary}
\newtheorem{Def}[lem]{Definition}
\newtheorem{prop}[lem]{Proposition}
\newtheorem{rem}[lem]{Remark}
\newtheorem*{thm*}{Theorem}
\newtheorem{ass}{Assumption}

\newcommand{\Z}{\mathbb{Z}}
\newcommand{\R}{\mathbb{R}}

\newcommand{\N}{\mathbb{N}}

\newcommand{\E}{\mathbb{E}}

\newcommand{\ep}{\varepsilon}

\newenvironment{disarray}{\everymath{\displaystyle\everymath{}}\array}{\endarray}

\begin{document}
\maketitle
\begin{abstract}
In this paper, we study random walks $g_n=f_{n-1}\cdots f_0$ on the group $\mathrm{Homeo}(S^1)$ of the homeomorphisms of the circle, where the homeomorphisms $f_k$ are chosen randomly, independently, with respect to a same probability measure $\nu$. We prove that under the only condition that there is no probability measure invariant by $\nu$-almost every homeomorphism, the random walk almost surely contracts small intervals. It generalizes what has been known on this subject until now, since various conditions on $\nu$ were imposed in order to get the phenomenon of contractions. Moreover, we obtain the surprising fact that the rate of contraction is exponential, even in the lack of assumptions of smoothness  on the $f_k$'s. We deduce various dynamical consequences on the random walk $(g_n)$: finiteness of ergodic stationary measures, distribution of the trajectories, asymptotic law of the evaluations, etc. The proof of the main result is based on a modification of the Ávila-Viana's invariance principle, working for continuous cocycles on a space fibred in circles.
\end{abstract}
\section{Introduction}
The objective of the paper is to study properties of  \textit{(left) random walks} on $\mathrm{Homeo}(S^1)$, that is to say long compositions $f_n\circ\cdots \circ f_0$ of homeomorphisms of the circle chosen randomly independently with respect to a same probability measure $\nu$. The study of independent random composition of transformations of a space $X$ is the theory of \textit{random dynamical systems} (RDS). They appear naturally for example in the theory of \textit{iterated forward systems} (IFS), when one wants to study the action of a finitely generated group or semigroup $G$: choosing $\nu$ uniform on a set of generators, the theory of RDS allows to study the properties of ``typical'' elements of $G$. The RDS also correspond to a natural family of skew-products on $X$: the ones of the form $(\omega,x)\mapsto (T\omega,f_{\omega}(x))$, where $T$ is a shift operator on a symbol space and $f_\omega$ only depends  on the first coordinate of $\omega$.\\

A standard starting point in order to study a random (or deterministic) dynamical system is the question of the dependence to the initial condition. In the context of RDS of homeomorphisms of the circle, the conclusion put in evidence by various results, is that in general the following alternative holds:

\begin{itemize}
\item either the iterated homeomorphisms preserve a common probability measure on the circle (which implies some ``determinism'' in the RDS)
\item or the RDS has the local contraction property: given any point of the circle, typical compositions of the homeomorphisms contract some neighbourhood of the point.
\end{itemize}
In the linear case (i.e. when the homeomorphisms are projective actions of elements of $SL_2(\R)$), that dichotomy is a well known result of H. Furstenberg \cite{Furstenberg} (and moreover, when the RDS has the local contraction property, these contractions are actually global and exponential). In the general case, there is variations of the precise assumptions and conclusions, but we can mainly distinguish two kinds of results:\\

--\underline{Smooth case}: In the case where the probability measure $\nu$ is supported on $\mathrm{Diff}(S^1)$, one can use the general theory of hyperbolic dynamical systems on manifolds. If the quantity $\int\log^+ \| f'\|_\infty d\nu(f)$ is finite, we can define \textit{Lyapunov exponents}. In this context, various results of hyperbolic dynamics (\cite{Crauel, Baxendale, Avila}) imply that if there is an invariant probability measure, then one can find a negative Lyapunov exponent in the system (one can see this as a non linear analogue of the Furstenberg's result stated above). Next, by Pesin theory (or even simpler arguments), one can deduce that the random dynamical system locally contracts, and even that the contractions are exponentially fast.\\ 

--\underline{Continuous case}: In the general case of the iteration of continuous homeomorphisms, the theory of hyperbolic dynamical systems, smooth by nature, does not apply any more. Though, coupling arguments of basic theory of the homeomorphisms of the circle with probabilistic arguments, it is still possible to obtain analogue results with no regularity assumption. The most canonical result (though the older one) of this kind is probably the following theorem of Antonov:
\begin{thm*}(Antonov) \cite{Antonov}\\
Let $f_1,\ldots,f_m$ be homeomorphisms of the circle preserving the orientation, such that the semigroup $G_+$ generated by $f_1,\ldots,f_m$ and the semigroup $G_-$ generated by $f_1^{-1},\ldots,f_m^{-1}$ both act \emph{minimally} on $S^1$ (i.e.~the orbit of every point is dense in the circle), and let $\nu$ be a non degenerated probability measure  on $\{1,\ldots,p\}$ (i.e. $\nu(\{i\})>0$ for $i=1,\ldots,p$). Then:
\begin{itemize}
\item Either for any initial conditions $x,y$ in $S^1$, for $\nu^\N$-almost every sequence $(i_n)_{n\geq 0}$, the distance between the trajectories $f_{i_n}\circ\cdots\circ f_{i_0}(x)$ and $f_{i_n}\circ\cdots\circ f_{i_0}(y)$ goes to $0$. (synchronization)
\item Either there exists a probability measure invariant by all the homeomorphisms $f_i$, and because of the minimality of $G_+$ it actually implies that $f_1,\ldots,f_p$ are simultaneously conjugated to rotations. (invariance)
\item Or there exists $\theta$ in $\mathrm{Homeo}_+(S^1)$ of finite order $p\geq 2$ commuting with all the $f_i$'s.(factorization)
\end{itemize}
\end{thm*}
\begin{rem}
When we are in the third case of Antonov Theorem, then one can factorize the system by identifying the points of the same orbit of $\theta$, in order to obtain a new topological circle, and  
homeomorphisms $\tilde{f}_1,\ldots,\tilde{f}_m$ of this circle induced by $f_1,\ldots,f_m$.

 We deduce that if $f_1,\ldots,f_m$ does not have a common invariant probability measure, then the random compositions of these homeorphisms satisfy the property of synchronization (first point of the alternative) up to some factorization (as described below).
 \end{rem}
 
As a consequence of Antonov's Theorem, it remains true that in absence of a common invariant probability measure we have the local contraction property. However, assuming no regularity for the iterated homeomorphisms has a price: additional structural assumptions are required and no speed of convergence is assured: the finiteness of the number of generators is only an assumption for convenience, and the proof of Antonov remains valid without this assumption. The minimality assumptions, though, are much deeper: the dynamics of a semigroup of $\mathrm{Homeo}(S^1)$ preserving some common interval is very different of the dynamics described in Antonov's Theorem. And if one considers a semigroup preserving two disjoint intervals, then one can check that in general, none of the alternatives of Antonov's Theorem are satisfied.

Variants of this theorem exist: let us cite for example \cite{Kleptsyn} where the authors proved (independently of Antonov) that synchronization occurs (first case of the previous theorem) under the additional assumption that $G_+$ contains a ``north-south'' homeomorphism, and \cite{Deroin-inter} where the  assumption of minimality is replaced by an assumption of symmetry ($G_+=G_-$).\\ \\

 The objective of the paper is to treat the study of a general random walk on $\mathrm{Homeo}(S^1)$. Adapting techniques coming from the hyperbolic theory in the continuous context, we show that the distinction between the regular and continuous cases described above is actually basically useless: there is no need to ask additional assumptions on a random walk on $\mathrm{Homeo}(S^1)$ to obtain the local contractions, and in fact, even the exponentially speed of contractions remains! Next we use this property of contraction to study deeply the behaviour of the random walk.
 
 We also deduce various results on the behaviour of random walks on $\mathrm{Homeo}(S^1)$. And the majority of these results actually holds for any random walk on a compact metric space satisfying the the local contraction property.
 \\

The key of the proof of the main result is to adapt the ideas of Ávila and Viana in \cite{Avila} and Crauel in \cite{Crauel} (who themselves used those of \cite{Ledrappier}) to establish that an invariance principle remains in the $C^0$-case:  but instead of using the Lyapunov exponents, we will use an another analogue quantity, which measures the exponential contractions as well, but which does not require derivability to be defined. That approach allows to obtain a criterion of the existence of exponential contractions for RDS of the circle, and more generally for any cocycle on a space fibred in circles, so that one can hope that this principle can also be useful in the study of non i.i.d. compositions of homeomorphisms of the circle.
\section{Statements of the results} \label{results}
\subsection{The main theorem}
Before stating our results, we need to formalize the notions of random walks and random dynamical systems:
\begin{Def}\label{randomwalk}
Let $(G,\circ)$ a topological semigroup.
\begin{itemize}
\item The \emph{random walk} generated by a probability measure $\nu$ on $G$ is the random sequence $\omega\mapsto (f_\omega^n)_{n\in\N}$ of elements of $G$ on the probability space $(\Omega,\mathbb{P})=(G^\N,\nu^\N)$, defined by: for $\omega=(f_n)_{n\in \N}$ in $\Omega$ and $n$ in $\N$, 
$$f_\omega^n=f_{n-1}\circ\cdots\circ f_0.$$
\item We denote by $G_+(\nu)$ the smallest closed sub-semigroup of $G$ containing the topological support of $\nu$. If $G_+(\nu)=G$, the random walk and the probability measure $\nu$ are said to be \emph{non degenerated} on $G$. It is equivalent to the fact that every open set of $G$ has positive probability to be reached by the random walk.

\item If $G$ acts on a space $X$ and if the probability measure $\nu$ is non degenerated on $G$, we say that $(G,\nu)$ is a \emph{random dynamical system} (RDS) on $X$. The skew-product associated to the RDS is the transformation $\hat{T}$ on $\Omega\times X$ defined by 
$$\hat{T}(\omega,x)=(T\omega,f_0(x)),$$
where $T$ is the shift operator on $\Omega$ and $f_0$ is the first coordinate of $\omega$.
\end{itemize}
\end{Def}
For a given random walk, we will always denote by $(\Omega,\mathbb{P})$ the associated probability space.

Obviously, any random walk on $\mathrm{Homeo}(S^1)$ is non degenerated on some sub-semigroup, namely $G_+(\nu)$. An interesting fact is that in the majority of the results that we will state, we obtain properties on the random walk depending only on assumptions on $G_+(\nu)$ and not on $\nu$ itself.\\

Here is the main theorem of the paper:
\begin{thm}\label{main} 
Let $\omega\mapsto (f_\omega^n)_{n\in\N}$ be a non degenerated random walk on a sub-semigroup $G$ of $\mathrm{Homeo}(S^1)$. Let us assume that $G$ does not preserve any probability measure on $S^1$ (i.e. there does not exist a probability measure invariant by every element of $G$). Then, for any $x$ in $S^1$, for $\mathbb{P}$-almost every $\omega$ in $\Omega$, there exists a neighbourhood $I$ of $x$ such that
$$\forall n\in \N, \mathrm{diam}(f_\omega^n(I))\leq q^n,$$
where $q<1$ depends on the random walk only.
\end{thm}
We can obtain the same result for random walks on a semigroup of continuous injective transformations of a compact interval $I$, since seeing $I$ as a part of $S^1$, such an injective map can be extended to a homeomorphism of the circle. Thus, in some sense, the surjectivity of the iterated transformations is not important. The injectivity, though, is primordial: one cannot hope to obtain a contraction phenomenon by iterating transformations of the circle homotopic to $z\mapsto z^2$.\\

In the case where the semigroup $G$ associated to a random walk on $\mathrm{Homeo}(S^1)$ preserves a probability measure $\mu$, then the topological support $K$ of $\mu$ is a compact minimal invariant by the group $\tilde{G}$ generated by $G$, and hence we have the standard trichotomy: $K$ is either $S^1$, a Cantor set or a finite set (see for exemple \cite{Navas}, Theorem 2.1.1). It is then standard that $\tilde{G}$ is conjugated to a group of isometries if $K=S^1$, and  semiconjugated to a group of isometries if $K$ is a Cantor set. This fact allows to obtain an interesting classification of the random walks on $\mathrm{Homeo}(S^1)$:

\begin{cor}\label{alternative}
Let $\omega\mapsto (f_\omega^n)_{n\in \N}$ be a non degenerated random walk on a sub-semigroup $G$ of $\mathrm{Homeo}(S^1)$. Then one (and only one) of the following possibilities occurs:
\begin{enumerate}[i)]
\item $G$ does not preserve a probability measure, and the random walk has the local contraction property in the sense given by Theorem \ref{main}.
\item  The random walk is semiconjugated to a random walk on the compact group $O_2(\R)$ (group of the isometries of the circle) acting minimally on $S^1$.
\item There is a finite set invariant by $G$.
\end{enumerate}
\end{cor}
On this form, the statement is very close to Furstenberg's one \cite{Furstenberg} in the linear case.
\subsection{General study of random walks acting on $\mathrm{Homeo}(S^1)$ }
In this section, we use Theorem \ref{main} as a main tool to understand the behaviour of a general random walk $\omega\mapsto (f_\omega^n)_{n\in\N}$ on $\mathrm{Homeo}(S^1)$.
\subsubsection{Distribution of the trajectories $n\mapsto f_\omega^n(x)$}
We interest in the typical distribution of the sequence $(f_\omega^n(x))_{n\in\N}$ for a given initial condition $x$. This problem is naturally related to the study of the \textit{stationary probability measures} of $\nu$, that is the probability measures $\mu$ on $S^1$ such that $\mathbb{P}\otimes \mu$ is invariant by the skew-product $\hat{T}$.  Such a probability measure always exists(we refer to \cite{Furman} or \cite{Kifer} for details). If the random walk is non degenerated on a subgroup of $\mathrm{Homeo}(S^1)$, it has been proved that in general, the stationary probability measure is unique (see \cite{Deroin-inter}). In the case of a general random walk on $\mbox{Homeo}(S^1)$, which is non degenerated on a semigroup only, it does not hold any more, but we prove that the number of \textit{ergodic} stationary probability measures (i.e.~extremal stationary probability measures) is necessarily finite, and that these probability measures give the typical distributions of the trajectories of the random walk:

\begin{thm}\label{distribution}
Let $\omega\mapsto (f_\omega^n)_{n\in \N}$ be a non degenerated random walk on a sub-semigroup $G$ of $\mathrm{Homeo}(S^1)$ with no finite orbit on $S^1$. Then:
\begin{itemize}
\item There is only a finite number of ergodic stationary probability measures $\mu_1,\ldots,\mu_d$. Their topological supports $F_1,\ldots,F_d$ are pairwise disjoints and are exactly the minimal invariant compacts of $G$.
\item For every $x$ in $S^1$, for $\mathbb{P}$-almost every $\omega$ in $\Omega$, there exists a unique integer $i=i(\omega,x)$ in $\{1,\ldots,d\}$ such that $F_i$ is exactly the set of accumulation points of the sequence $(f_\omega^n(x))_{n\in\N}$, and then we have
$$\frac{1}{N}\sum_{n=0}^{N-1}\delta_{f_{\omega}^n(x)}\xrightarrow[n\to+\infty]{} \mu_i$$
in the weak-$*$ topology of $C(S^1,\R)^*$.
\end{itemize}
\end{thm}
Note that in this theorem, we relaxed the condition ``no invariant probability measure'' to ``no finite orbit''.\\ 

As a direct consequence of this theorem, we obtain that the stationary probability measure is unique when the action is minimal:
\begin{cor}\label{minimal ergodique}
A  non degenerated random walk on a sub-semigroup $G$ of $\mathrm{Homeo}(S^1)$ acting minimally on $S^1$ is uniquely ergodic, i.e. it admits a unique stationary probability measure.
\end{cor}
At our knowledge, this fact was never proved in full generality: until now some additional assumption (smoothness, backward minimality, symmetry...) was required to obtain the unique ergodicity. And actually, we obtain  a slightly stronger corollary: the action of any random walk of $\mathrm{Homeo}(S^1)$ restricted to a minimal invariant compact $F$ is uniquely ergodic: if there is no finite orbits, that is a consequence of Theorem \ref{distribution}, and if there is a finite orbit, then $F$ is necessarily finite and the unique ergodicity follows easily).

\subsubsection{Law of probability of $\omega\mapsto f_\omega^n(x)$}
We focus now in the law of the random variables $X_n^x:\omega\mapsto f_\omega^n(x)$ for any given initial condition $x$ and a large integer $n$, and asking whether the law of $X_n^x$ converges to some limit distribution when $n$ becomes large. \\

The sequence $(X_n^x)_{n\in\N}$ is a Markov chain. A natural obstruction to the convergence of the laws of a Markov chain are the ``periodic configurations'', where there exists  subspace of phase states whose the return times are multiple of a fixed integer larger than $2$. (For exemple in our context, if it exists two disjoints closed sets $F_1$ and $F_2$ such that the generators of the semigroup send $F_1$ into $F_2$ and $F_2$ into $F_1$, then clearly the distribution of $X_n^x$ strongly depends on the parity of $n$.). That leads us to the following definition of \textit{aperiodicicity}:
\begin{Def}\label{indeco}
A random walk $\omega\mapsto (f_\omega^n)_{n\geq 0}$ on $\mathrm{Homeo}(S^1)$ generated by a probability $\nu$ is said to be \emph{aperiodic} if there does not exist a finite number $p\geq 2$ of pairwise disjoints closed subsets $F_1,\ldots,F_p$ of $S^1$ such that for $\nu$-almost every homeomorphism $g$, $g(F_i)\subset F_{i+1}$ for $i=1,\ldots,p-1$ and $g(F_p)\subset F_1$.
\end{Def}
\begin{rem}\label{minaper}
If the action of $G$ is minimal, the random walk is necessarily aperiodic since otherwise, $S^1$ would be a non trivial finite union of pairwise disjoints closed subsets.
\end{rem}
The next theorem states that for random walks with no invariant probability measure, the only obstruction to the convergence in law of $X_n^x$ is the one described above:
\begin{thm}\label{law}
Let $\omega\mapsto (f_\omega^n)_{n\in\N}$ be a non degenerated random walk on a sub-semigroup $G$  of $\mathrm{Homeo}(S^1)$ with no invariant probability measure on $S^1$, and such that the random walk is aperiodic. Then, for every $x$ in $S^1$, denoting by $\mu_{n}^x$ the law of the random variable $X_n^x:\omega\mapsto f_\omega^n(x)$, we have the convergence in law
$$\mu_n^x\xrightarrow[n\to +\infty]{} \mu^x,$$
where $\mu^x$ is a stationary probability measure of the random walk. Moreover, the convergence is uniform in $x$ in the sense that for any continuous test function $\varphi:S^1\rightarrow \R$,
$$\sup_{x\in S^1} \left|\int_{S^1} \varphi d\mu_n^x-\int_{S^1} \varphi d\mu^x\right|\xrightarrow[n\to +\infty]{} 0$$
\end{thm}

In particular, as a consequence of this theorem, Remark \ref{minaper} and Corollary \ref{minimal ergodique}:
\begin{cor}
	Let $(f_\omega^n)_{n\in\N}$ be a non degenerated random walk on a sub-semigroup $G$ of $\mathrm{Homeo}(S^1)$,  acting minimally on $S^1$ and with no invariant probabiity measure on $S^1$. Then, with the same notations as Theorem \ref{law}, we have for every $x$ in $S^1$:
	$$\mu_n^x\xrightarrow[n\to +\infty]{} \mu,$$
	where $\mu$ is the unique stationary probability measure of the random walk.
\end{cor}

\subsubsection{Behaviour of typical homeomorphisms $x\mapsto f_\omega^n(x)$}
Finally, for $\omega$ typical we focus in the behaviour of the homeomorphisms $f_\omega^n$ when $n$ become large.
\begin{thm}\label{Lejan-Antonov}
Let $\omega\mapsto (f_\omega^n)_{n\geq 0}$ be a non degenerated random walk on a sub-semigroup $G$ of $\mathrm{Homeo}(S^1)$, such that $G$ does not preserve a common invariant probability measure on $S^1$. Then, there exists a finite number $p$ of measurable functions $\sigma_1,\ldots,\sigma_p:\Omega\rightarrow S^1$ such that: for $\mathbb{P}$-almost every $\omega$ in $\Omega$, for every closed interval $I$ included in $S^1-\{\sigma_1(\omega),\ldots,\sigma_p(\omega)\}$, $\mathrm{diam}(f_\omega^n(I))\xrightarrow[n\to+\infty]{}0$ exponentially fast.
\end{thm}
It is a global version of Theorem \ref{main}, proving that for $n$ large, the typical homeomorphisms  $f_\omega^n$ are close to be ``staircase maps'', with a constant finite number of stairs.

It is also intersting to compare this result with Antonov Theorem stated in the introduction. The hypotheses of Antonov Theorem are stronger than the ones of Theorem \ref{Lejan-Antonov}, since it assume forward and backward minimality and that the homeomorphisms preserve the orientation. In counterpart, the conclusion of Antonov Theorem is in some sense stronger: it does not give the exponential speed of the contractions, but gives a more precise structure: it say that the applications $\sigma_1,\ldots,\sigma_p$ given by Theorem \ref{Lejan-Antonov} are on the form.
$$\{\sigma_1,\sigma_2\ldots, \sigma_p\}=\{\sigma_1,\theta\circ\sigma_1\ldots, \theta^{p-1}\circ\sigma_1\},$$
 where $\theta$ is a homeomorphism of order $p$ commuting with all the elements of $G$.
But as we said in the introduction, such a rigid conclusion cannot hold in general in a non minimal context.\\

\subsection{Property of synchronization}
In this section, we want to characterize in which situation the action of a random walk on the circle has the property of \textit{synchronization}, which means that for any couple of initial conditions $x$ and $y$, for almost every realization of the random walk, the distance between the corresponding trajectories of $x$ and $y$ tends to $0$. This property of synchronization has been studied in \cite{Furstenberg} in the linear case, and for example in \cite{Homburg, Kaijser, Kleptsyn, Kleptsyn2} in non linear cases.

\begin{Def}
If $(X,d)$ is a metric space, we say that a random walk $\omega\mapsto (f_\omega^n)_{n\in\N}$ acting on $X$ is  \emph{synchronizing} if for every $x$, $y$ in $X$, for almost every $\omega$,
$$d(f_\omega^n(x),f_\omega^n(y))\xrightarrow[n\to+\infty]{} 0.$$
We say that it is \emph{exponentially synchronizing} if the previous convergence is exponentially fast.
\end{Def}
In the context of random walks acting on the circle, we prove that the synchronization is equivalent to the \textit{proximality} of the action. We recall that the action of a semigroup $G$ to a metric space $(X,d)$ is proximal if for every $x,y$ in $X$, there exists a sequence $(g_n)_{n\in\N}$ in $G$ such that 
$$d(g_n(x),g_n(y))\xrightarrow[n\to +\infty]{} 0.$$
\begin{thm}
\label{synchronizing}
Let $\omega\mapsto (f_\omega^n)_{n\in\N}$ be a non degenerated random walk on a sub-semigroup $G$ of $\mathrm{Homeo}(S^1)$ without a common fixed point. Then the following properties are equivalent:
\begin{enumerate}[i)]
\item  The random walk is exponentially synchronizing. 
\item  The random walk is synchronizing.
\item The action of $G$ on $S^1$ is proximal.
\end{enumerate}
\end{thm}

It allows for example to retrieve the main result of \cite{Kleptsyn} in a non minimal context and with an exponential speed of convergence:
\begin{cor}\label{Kleptsyn}
$\omega\mapsto (f_\omega^n)_{n\in\N}$ a non degenerated random walk on a sub-semigroup $G$ of $\mathrm{Homeo}(S^1)$ such that:
\begin{itemize}
\item  $G$ contains a map $g_0$ with exactly $2$ fixed points $a$ and $b$, one attractive, one repulsive.
\item  None of the sets $\{a\}$, $\{b\}$, $\{a,b\}$ is invariant by the semigroup $G$.
\end{itemize}
Then the random walk is exponentially synchronizing. 
\end{cor}
That corollary follows rather easily from Theorem \ref{synchronizing}: for any $x,y$ in $S^1$, one can find $h$ in $G$ such that $h(x)$ and $ h(y)$ are distinct from the repulsive fixed point of $g_0$, so that $\mathrm{dist}(g_0^n\circ h(x),g_0^n\circ h(y))\xrightarrow[n\to +\infty]{} 0$, which prove the proximality and we can apply Theorem \ref{synchronizing}. The details are left to the interested reader.\\ \\

An other application deals with the \textit{robustness} of the property of synchronization (that is to say the persistence of the property to small perturbations): with Theorem \ref{synchronizing}, we can prove that the property of synchronization is robust among the semigroups of homeomorphisms without a common fixed point. We restrict ourselves to the case of finitely generated semigroups to avoid to manipulate intricate topologies on sets of semigroups/random walks.

\begin{cor}\label{robust}
Consider a non degenerated random walk $\omega\mapsto (f_\omega^n)$ on a sub-semigroup $G$ of $\mathrm{Homeo}(S^1)^d$ generated by $d$ homeomorphisms  of the circle $f_1,\ldots,f_d$ without common fixed points, and assume that $\omega\mapsto (f_\omega^n)$ is synchronizing. Then there exists a neighbourhood $\mathcal{V}$ of $(f_1,\ldots,f_d)$ in $\mathrm{Homeo}(S^1)^d$
such that for any $d$-tuple $(\tilde{f}_1,\ldots,\tilde{f}_d)$ in $\mathcal{V}$, any non degenerated random walk $\omega\mapsto (f_\omega^n)$ on the semigroup $\tilde{G}$ generated by $\{\tilde{f}_1,\ldots,\tilde{f}_d\}$ is (exponentially) synchronizing.
\end{cor} 
It is natural to ask whether the property of synchronization is generic, but it easy to see that it is not the case: if $I$ is an open interval, the property 
$$\forall k\in \{1,\ldots,d\},\, \overline{f_k(I)}\subset I$$
is robust, and the existence of two disjoints such intervals is an obstruction to the synchronization. However, in the case of a non degenerated random walk on subgroups of $\mathrm{Homeo}_+(S^1)$, Antonov's Theorem holds (see \cite{Deroin-inter}), and hence in this case, the property of synchronization is generic, because the other alternatives (existence of a common invariant probability measure or existence of a non trivial homeomorphism in the centralizer of the group) are degenerated properties. Combining this remark with Corollary \ref{robust}, we obtain the following conclusion:

\begin{cor}\label{generic}
Let $d$ be an integer larger than $1$. Then there exists an open dense subset $\mathcal{U}$ of $\mathrm{Homeo}_+(S^1)^d$ such that for every $(f_1,\ldots,f_d)$ in $\mathcal{U}$, any non degenerated random walk on the group generated by $\{f_1,\ldots,f_d\}$ is exponentially synchronizing.
\end{cor}

\subsection{Random dynamical systems on $[0,1$]}
We conclude by the study of the iterations of continuous injective transformations of an interval. For exemple, we can apply our results to obtain:

\begin{cor}\label{example}
	Let $\omega\mapsto (f_\omega^n)_{n\in\N}$ be a non degenerated random walk on a semigroup $G$ of injective continuous functions from $[0,1]$ into itself, and let us assume that
	$$\bigcap_{g\in G} g([0,1])=\emptyset.$$
	Then there exists $q<1$ such that for $\mathbb{P}$-almost every $\omega$:
	$$\forall n\in \N,\, \mathrm{diam}(f_\omega^n([0,1]))\leq C q^n$$
	for some constant $C=C(\omega)$.
\end{cor}
The assumption $\bigcap_{g\in G} g([0,1])=\emptyset$ is weak (and is actually equivalent to the conclusion if $G$ does not fix any point of $I$): for example, if you iterate randomly two continuous injective functions $f_1,f_2:I\rightarrow I$ such that $f_1$ has only one fixed point $c$, and $f_2(c)\not=c$, then the corollary applies, that is to say that random compositions of $f_1$ and $f_2$ almost surely contract the whole interval $[0,1]$ exponentially fast.\\
\begin{rem}
	It is actually possible to prove Corollary \ref{example} by a straight elementary proof, using the ideas of \cite{Yuri}.
\end{rem}

The previous corollary does not apply if we iterate homeomorphisms of the interval. The techniques of the paper does not seem to be sufficient to treat such a random walk in a general exhaustive way. However we can still adapt our techniques to get some partial information. Here is a variation of our main theorem in this context:

\begin{cor}\label{mainR}
	Let $\omega\mapsto (f_\omega^n)_{n\geq 0}$ be a non degenerated random walk on a sub-semigroup $G$ of $\mathrm{Homeo}([0,1])$, such that:
	\begin{itemize} 
	\item there does not exists a non trivial subinterval of $(0,1)$ invariant by $G$.
	\item there exists at least one probability measure $\mu$ on $(0,1)$ which is stationary for the random walk.
	\end{itemize}
	 Then, for every $x$ in $\R$ there exists a neighbourhood $I$ of $x$ such that
	$$\forall n\in \N,\, \mathrm{diam}(f_\omega^n(I))\leq q^n,$$
	where $q<1$ does not depend on $x$.
\end{cor}

\begin{rem}
	From the proof of this corollary one can notice that if the second assumption is satisfied but not the first one, then the conclusion of the statement still holds if we restrict $x$ to belong to the invariant interval $I=(\inf (\mathrm{supp}(\mu)),\sup (\mathrm{supp}(\mu)))$
\end{rem}

This theorem gives the phenomenon of local contractions under the existence of a stationary probability measure. With some additional work, one can hope to deduce various dynamical properties from it as we do in this paper in the case of the circle. As an example, let us state the following corollary, answering by the affirmative to a question of B. Deroin in \cite{Deroin-cours}: ``If $f,g$ are increasing diffeomorphisms of $[0,1]$, and if the Lebesgue measure is stationary (for $\nu=\frac{\delta_f+\delta_g}{2}$), is it necessarily the only stationary probability measure without atoms?''\\

\begin{cor}\label{torchder}
If a random walk on $\mathrm{Homeo}_+([0,1])$ admits a stationary probability measure on $(0,1)$ with total support, then it is the only one. In particular, any random walk on $\mathrm{Homeo}_+([0,1])$ acting minimally on $(0,1)$ admits at most one stationary probability measure on $(0,1)$.
\end{cor}

 The existence of a stationary probability measure for a random walk on $\mathrm{Homeo}([0,1])$ (other that convex combinations of $\delta_0$ and $\delta_1$ ) can be ensured if the extremities $0$ and $1$ ´´repulse'' the dynamics of the random walk. One can check for exemple that a stationary probability measure exists if the random walk is generated by a probability $\nu$ on $Diff^1_+([0,1])$ whose finite support, such that $\int \log |f'(c)|d\nu(f)>0$ for $c=1, 2$.
 
 However, without such an additional assumption, in general such a measure does not exists.
 For exemple, when the random walk is symmetric in the sense that the associated probability measure $\nu$ is invariant under the transformation $g\mapsto g^{-1}$, it is proved in \cite{Deroin-homeo} that there is no stationary probability measure. Thus, Corollary \ref{mainR} does not apply in this case. But it is interessant to notice that \cite{Deroin-homeo} develops techniques to obtain a good understanding of the random walk in this particular case where ours methods do not apply. In consequence one could hope that by adapting these techniques and those of this paper it would be possible to manage the study of a general random walk on $\mathrm{Homeo}([0,1])$.
\subsection{Scheme of the paper}
The paper is organized as follows: 
\begin{itemize}
	
\item in Section \ref{sec-inv} we present the core argument of our results: an invariance principle for a general skew-product $\hat{T}$ on a space $\Omega\times S^1$ stating that either there is a phenomenon of contractions in the dynamics of $\hat{T}$ on the fibres, either ``there is something invariant''. Applying the principle to the specific case where $\hat{T}$ is associated to a random walk, we obtain Theorem \ref{main}, and one can hope that it can also be used in non independant contexts. 

\item In Section \ref{sec-consequences}, we state various ergodic properties of the random dynamical systems on compact metric sapces satisfying the property of local contractions. (This section can be read indepedently of the others)

\item In Section \ref{sec-proofs}, we deduce the proofs of the other theorems stated in the introduction by combining the results of Sections \ref{sec-inv} and \ref{sec-consequences}.

\end{itemize}
\section{An invariance principle}\label{sec-inv}
The objective of this part is to prove an invariance principle in the spirit of the works of Ledrappier \cite{Ledrappier}, Crauel \cite{Crauel} and Ávila-Viana \cite{Avila} for one-dimensional cocycles without regularity (except the continuity).\\

Let $(\Omega,\mathcal{F},\mathbb{P})$ be a probability space and $T:\Omega \rightarrow \Omega$ be a $\mathbb{P}$-invariant transformation. We look at the skew products on $\Omega\times S^1 $ extending $T$, that is the measurable transformations $\hat{T}$ of the form $(\omega,x)\mapsto (T\omega, f_\omega(x))$, where $f_\omega\in \mathrm{Homeo}(S^1)$. For $\omega$ in $\Omega$, we will use the notation
$$f_\omega^n=f_{T^{n-1}\omega}\circ\cdots\circ f_\omega,$$
so that the iterates of $\hat{T}$ are given by $\hat{T}^n(\omega,x)=(T^n\omega,f_\omega^n(x))$.

\subsection{Lyapunov exponent and exponent of contraction}
Let us recall the definition of the \emph{Lyapunov exponents} of $\hat{T}$ when $f_\omega$ is smooth:
\begin{Def}
If $f_\omega\in \mbox{Diff}(S^1)$, then the
{Lyapunov exponent} of $\hat{T}$ at a point $(\omega,x)\in\Omega\times S^1$ is defined as $$\lambda(\omega,x)=\lim_{n\to+\infty}\frac{\log |(f_\omega^n)'(x)|}{n},$$
if the limit exists. If $\hat{\mu}$ is a $\hat{T}$-invariant probability measure such that $(\omega,x)\mapsto \log |f_\omega'(x)|$ is $\hat{\mu}$-integrable, then the Lyapunov exponent is well defined $\hat{\mu}$-almost everywhere, constant if $\hat{\mu}$ is ergodic, and the Lyapunov exponent of $\hat{\mu}$ is defined as
$$\lambda(\hat{\mu})=\int_{\Omega\times S^1} \lambda(\omega,x)d\hat{\mu}(\omega,x).$$	
\end{Def}

The Lyapunov exponent $\lambda(\omega,x)$ measures the exponential rate of contraction of $(f_\omega^n)$ at the neighbourhood of $x$. In order to have analogue informations without assuming that $f_\omega\notin \mbox{Diff}^1(S^1)$, we define the following \emph{exponent of contraction}:
\begin{Def}
The exponent of contraction of $\hat{T}$  at the point $(\omega,x)$ is the non positive quantity
$$\lambda_{con}(\omega,x)=\varlimsup_{y\to x}\varlimsup_{n\to +\infty}\frac{\log(\mathrm{dist}(f_\omega^n(x),f_\omega^n(y))}{n}.$$ 
If $\hat{\mu}$ is a $\hat{T}$-invariant probability measure, the exponent of contraction of $\hat{\mu}$ is defined as 
$$\lambda_{con}(\hat{\mu})=\int_{\Omega\times S^1} \lambda_{con}(\omega,x)d\hat{\mu}(\omega,x).$$
\end{Def}
Note that $\lambda_{con}$ is $\hat{T}$ is $\hat{T}$-invariant, so that $\lambda_{con}$ is constant $\hat{\mu}$-almost everywhere if $\hat{\mu}$ is ergodic.

The exponent of contraction has the advantage over Lyapunov exponents that it does not need any assumption of differentiability. As a counterpart, the information provided by this exponent is slightly less precise than the one provided by the Lyapunov exponents, because it only measures the contraction of the cocycle, not the expansion, and actually  the maximal contraction only, so that one cannot hope miming an Oselede\v{c}/Pesin's theory with this naive definition in dimension larger than one. In dimension one, though, this exponent of contraction is a perfect tool to generalize the notion of Lyapunov exponent. We have indeed in this case a simple relation between Lyapunov exponent and exponent of contraction:
\begin{prop}\label{expoegal}
Let $(\Omega,\mathcal{F},\mathbb{P})$ be a probability space and let $\hat{T}: (\omega,x)\mapsto (T\omega, f_\omega(x))$ be a measurable transformation of $\Omega \times S^1$ with $f_\omega\in \mathrm{Diff^1}(S^1)$, such that the function $(\omega,x)\mapsto \log |f_\omega'(x)|$ is bounded. Then, for every $\hat{T}$-invariant ergodic probability measure $\hat{\mu}$, we have
$$\lambda_{con}(\hat{\mu})=\inf(\lambda(\hat{\mu}),0),$$
where $\lambda(\hat{\mu})$ is the Lyapunov exponent associated to $\hat{\mu}$:
$$\lambda(\hat{\mu})=\int_{\Omega\times S^1}\log |f_\omega '(x)|d\hat{\mu}(\omega,x).$$
\end{prop}
\begin{proof}
The inequality $\lambda_{con}(\hat{\mu})\leq 0$ is trivial. And as noticed in \cite{Crauel} Proposition 2.6, one can adapt the techniques of Pesin on stable manifolds, to obtain the inequality $\lambda_{con}(\hat{\mu})\leq \lambda(\hat{\mu})$ (see also \cite{Lejan} for a proof in the particular case of independent compositions of diffeomorphisms). So from now on, we focus on proving the converse inequality $\lambda_{con}(\hat{\mu})\geq \inf(\lambda(\hat{\mu}),0).$\\

We assume that $\lambda_{con}(\hat{\mu})<0$. Let $(\omega,x)$ be a point of $\Omega\times S^1$ such that 
$$\lambda(\hat{\mu})=\lambda(\omega,x)=\lim_{n\to+\infty}\frac{\log |(f_\omega^n)'(x)|}{n}$$ 
and 
 $$\lambda_{con}(\hat{\mu})=\lambda_{con}(\omega,x)=\varlimsup_{y\to x}\varlimsup_{n\to +\infty}\frac{\log(\mathrm{dist}(f_\omega^n(x),f_\omega^n(y))}{n},$$
 and let $\ep>0$. If $y$ is close enough to $x$, then we have
$$\forall n\in \N, \mbox{dist}(f_\omega^n(x),f_\omega^n(y))\leq \ep.$$
Then, denoting by $\alpha_\omega(\cdot)$ the modulus of continuity of $\log |f_\omega'|$, we have for any $z_1$, $z_2$ in $[x,y]$,
$$\log |(f_\omega^n)'(z_1)|-\log|(f_\omega^n)'(z_2)|= \sum_{k=0}^{n-1}\log |f_{T^k\omega}'(f_\omega^k(z_1))|-\log |f_{T^k\omega}'(f_\omega^k(z_2))|\leq \sum_{k=0}^{n-1}\alpha_{T^k\omega}(\ep).$$
In particular, by the mean value equality,
$$\frac{\log |(f_\omega^n)'(x)|}{n}\leq \frac{1}{n}\log\left(\frac{\mbox{dist}(f_\omega^n(x),f_\omega^n(y))}{\mbox{dist}(x,y)}\right)+
\frac{1}{n}\sum_{k=0}^{n-1}\alpha_{T^k\omega}(\ep).$$

If $\omega$ is a Birkhoff point of $\alpha_\cdot(\ep)$, we deduce by letting $n$ tend to $+\infty$ and $y$ to $x$ that 
$$\lambda(\hat{\mu})\leq \lambda_{con}(\hat{\mu})+\int_{\Omega} \alpha_{\omega '}(\ep) d\mathbb{P}(\omega ')$$
Since $\alpha_{\omega '}(\ep)$ tends to $0$ as $\ep\to 0$ and is uniformly bounded, by dominated convergence we obtain that $\lambda(\hat{\mu})\leq \lambda_{con}(\hat{\mu})$.
\end{proof}

\subsection{The invariance principle statement}
 Let us state the main theorem of the section.
\begin{thm}[Invariance principle]\label{invariance}~

Let $(\Omega,\mathcal{F},\mathbb{P})$ be a standard Borel space, with $\mathbb{P}$ a probability measure, and let $\hat{T}: (\omega,x)\mapsto (T\omega, f_\omega(x))$ be  a measurable transformation of $\Omega \times S^1$ with $f_\omega\in \mathrm{Homeo}(S^1)$. Then, for every $\hat{T}$-invariant probability measure $\hat{\mu}$ of the form $d\hat{\mu}(\omega,x)=d\mu_\omega(x)d\mathbb{P}(\omega)$, we have the following alternative:
\begin{itemize}
\item either $\lambda_{con}(\hat{\mu})<0$  (contraction),
\item or for $\mathbb{P}$-almost every $\omega$, $\mu_{T\omega}=(f_\omega)_*\mu_{\omega}$ (invariance).
\end{itemize}
\end{thm}
\begin{rem}
	In the case that $f_\omega$ is smooth, by Proposition \ref{expoegal} we otain the known fact that the Lyapunov exponent of $\hat{\mu}$ is negative unless maybe if we have the ''deterministic relation'' $\mu_{T\omega}=(f_\omega)_*\mu_{\omega}$. In particular the Lyapunov exponent of a stationary probability measure of a random walk on $\mbox{Diff}^1(S^1)$ is negative unless the stationary measure is actually invariant)
	\end{rem}

When the transformation $T$ of $\Omega$ is invertible, the relation $\mu_{T\omega}=(f_\omega)_*\mu_{\omega}$ is only a reformulation of ``$\hat{\mu}$ is $\hat{T}$-invariant'', so that the invariance principle as we stated it only gives information in non-invertible contexts (it is possible though to get an invariance principle in an invertible context, applying the theorem to a modified system, see \cite{Ledrappier}).\\

The following subsections 3.3, 3.4 and 3.5 are dedicated to the proof of Theorem \ref{invariance}. We will keep the notations of the statement in these subsections.
\subsection{Fibred Jacobian and fibred entropy}
Following the ideas of \cite{Ledrappier, Crauel, Avila}, we define the fibred entropy of $\hat{\mu}$ as follows:

\begin{Def}
The \emph{fibred Jacobian} $J=J(\hat{\mu}):\Omega\times S^1\to\R$ of $\hat{\mu}$ is defined by the expression
$$J(\omega,x)=\frac{d(f_\omega^{-1})_*\mu_{T\omega}}{d\mu_\omega}(x),$$
where the derivative is  taken in the Radon-Nikodym sense. The \emph{fibred entropy} $h(\hat{\mu})$ of $\hat{\mu}$ is defined as
$$h(\hat{\mu})=\begin{cases}
\displaystyle-\int_{\Omega\times S^1}\log J\, d\hat{\mu}&\text{if }\log J\in L^1(\Omega\times S^1,\hat{\mu}),\\

+\infty&\text{otherwise}.
\end{cases}$$
\end{Def}
By definition, the mapping $x\mapsto J(\omega,x)$ is the derivative of Radon-Nikodym of the measure $(f_\omega^{-1})_*\mu_{T\omega}$ against $\mu_\omega$, that is the $\mu_\omega$-integrable function such that we can write
\begin{equation}\label{Radon}d(f_\omega^{-1})_*\mu_{T\omega}(x)=J(\omega,x)\,d\mu_\omega(x)+d\tilde{\mu}_\omega(x)\end{equation}
where $\tilde{\mu}_\omega$ is singular with respect to $\mu_\omega$.

Let us state a classical general fact of geometric measure theory which allows to see a Radon-Nikodym derivative as, in some sense, a standard derivative: 
\begin{prop}\label{Besicovitch}
Let $\mu$ be a probability measure on $S^1$, and $\nu$ be any measure on $S^1$. Then:
\begin{enumerate}[i)]
\item For $\mu$-almost every $x$ in $S^1$, 
$$\frac{d\nu}{d\mu}(x)=\lim_{I\ni x,\, \mathrm{diam}(I)\to 0}\frac{\nu(I)}{\mu(I)}$$
(here and in the sequel, $I$ represents an interval of $S^1$).
\item  Denoting $q^*(x)=\sup_{I\ni x}\frac{\nu(I)}{\mu(I)}$,
$$\int_{S^1}\log^+ q^*(x)d\mu(x)\leq 2\nu(S^1).$$
\end{enumerate}
\end{prop}
This proposition is standard if $\nu$ is the Lebesgue measure, as a consequence of Vitali's covering Lemma, and as noticed in \cite{Ledrappier}, the proof adpapts for any measure $\nu$ if we use Besicovitch's covering Lemma instead of Vitali's.\\

The key of the proof of Theorem \ref{invariance} is to see the entropy $h(\hat{\mu})$ in two different ways. 
\begin{itemize}
\item \emph{Firstly},
one can see $h(\hat{\mu})$ as a quantity measuring in average how much $(f_\omega^{-1})_*\mu_{T\omega}$ differs from $\mu_\omega$, and obtain the following fact justifying that $h(\hat{\mu})$ deserves its appellation of entropy:
\begin{prop}\label{entropy-positive}
We have the inequality
$$h(\hat{\mu})\geq 0,$$
with equality if and only if for $\mathbb{P}$-almost every $\omega$, $\mu_{T\omega}=(f_\omega)_*\mu_{\omega}$.
\end{prop}
\item  \emph{Secondly}, one can use Proposition \ref{Besicovitch} to see the Jacobian term $J(\omega,x)$ as a kind of derivative for some geometry: for $\hat{\mu}$-a.e. $(\omega,x)$ in $\Omega\times S^1$,
$$J(\omega,x)=\lim_{y\to x}\frac{\mu_{T\omega}([f_\omega(x),f_\omega(y)])}{\mu_\omega([x,y])}.$$
 It is then possible to think of $-h(\hat{\mu})$ as a kind of Lyapunov exponent, and obtain:
\begin{prop}\label{ineq-entropy}
We have the inequality
$$\lambda_{con}(\hat{\mu})\leq -h(\hat{\mu})$$
\end{prop}
\begin{rem}
With a slighter effort, we could actually prove the more precise inequality $\lambda_{con}(\hat{\mu})\leq -\frac{h(\hat{\mu})}{d(\hat{\mu})}$, for a good definition of the \emph{fibred dimension} $d(\hat{\mu})$, which would thus belongs to the big family of inequalities relating Lyapunov exponent, entropy and dimension (see for example \cite{Young, Hu, Ledrappier2}).
\end{rem}
\end{itemize}

It is clear that Theorem \ref{invariance} is a direct consequence of Propositions \ref{entropy-positive} and \ref{ineq-entropy}. Let us begin by proving Proposition \ref{entropy-positive} (the easy part):

\begin{proof}[Proof of Proposition \ref{entropy-positive}]
As a consequence of \eqref{Radon}, 
$$\int_{\Omega\times S^1} J\, d\hat{\mu}=\int_{\Omega}\int_{S^1} J(\omega,x)\,d\mu_\omega(x)d\mathbb{P}(\omega)
\leq \int_{\Omega}\int_{S^1} d\mu_{T\omega}(x)d\mathbb{P}(\omega)=1,$$
hence, by Jensen inequality,
\begin{equation}\label{Jensen}-h(\hat{\mu})=\int_{\Omega\times S^1} \log J\, d\hat{\mu}\leq \log \int_{\Omega\times S^1} J\, d\hat{\mu} \le 0,\end{equation}
so that $h(\hat{\mu})$ is non negative.

Moreover, if $h(\hat{\mu})=0$, then the Jensen inequality (\ref{Jensen}) is in fact an equality, so that $J=1$ $\hat{\mu}$-almost everywhere. Thus, replacing it in \eqref{Radon}, we deduce that for $\mathbb{P}$-almost every $\omega$, $(f_\omega^{-1})_*\mu_{T\omega}=\mu_\omega$, hence $\mu_{T\omega}=(f_\omega)_*\mu_\omega$.
\end{proof}

We focus now on the proof of Proposition \ref{ineq-entropy}. In the following subsection, we dismantle the problem and leave the core arguments for  a separated treatment in the section afterwards.
\subsection{Preliminaries: reduction of the problem}
The objective of this subsection is to check that it is enough to prove Proposition \ref{ineq-entropy} in the case where we have some useful additional properties on $\hat{\mu}$, namely:
\begin{itemize}
\item $\hat{\mu}$ is ergodic.
\item None of the probability measures $\mu_\omega$ has atoms on $S^1$.
\end{itemize}

The reduction of the problem to the ergodic case is done by ergodic disintegration: let us write
$$\hat{\mu}=\int \hat{\mu}_\alpha\, d\alpha$$
with $\hat{\mu}_\alpha$ ergodic and $d\alpha$ some probability measure on the set of ergodic probability measures. Then, writing $d\hat{\mu}_\alpha=d\mu_{\alpha,\omega}d\mathbb{P}_\alpha$ and setting $$J_\alpha(\omega,x)=\frac{d(f_\omega^{-1})_*\mu_{T\omega,\alpha}}{d\mu_{\omega,\alpha}}(x)$$
the Jacobian associated to $\hat{\mu}_\alpha$, we have that $J_\alpha=J$ $\hat{\mu}_\alpha$-almost everywhere, and as a consequence,

$$h(\hat{\mu})=-\iint_{\Omega\times S^1} \log J\, d\hat{\mu}_\alpha d\alpha=-\iint_{\Omega\times S^1} \log J_\alpha\, d\hat{\mu}_\alpha d\alpha=\int h(\hat{\mu}_\alpha)\,d\alpha.$$

Moreover, we also have
$$\lambda_{con}(\hat{\mu})= \iint_{\Omega\times S^1} \lambda_{con}(\omega,x)\,d\hat{\mu}_\alpha(\omega,x)\,d\alpha=\int \lambda_{con}(\hat{\mu}_\alpha)\,d\alpha,$$
hence the inequality to prove is $\int \lambda_{con}(\hat{\mu}_\alpha)\,d\alpha\leq -\int h(\hat{\mu}_\alpha)\,d\alpha$, which follows from the inequalities in the ergodic case $\lambda_{con}(\hat{\mu}_\alpha)\leq -h(\hat{\mu}_\alpha)$.\\

Thus from now on, we assume that $\hat{\mu}$ is ergodic. The case where $\mu_\omega$ has atoms is treated by the following general lemma:

\begin{lem}\label{atoms}
If $\hat{\mu}$ is ergodic, and if the set $\{\omega\in \Omega| \mu_{\omega} \mbox{ has atoms}\}$ has $\mathbb{P}$-positive probability, then there exists a family $(E(\omega))_{\omega\in\Omega}$ of finite subsets of $S^1$, all of them with same cardinal $d$, such that for $\mathbb{P}$-almost every $\omega$ in $\Omega$, $f_\omega(E(\omega))=E(T\omega)$ and $\displaystyle\mu_\omega=\frac{1}{d}\sum_{x\in E(\omega)}\delta_x$.
\end{lem}
\begin{rem}
The proof does not use the structure of $S^1$ so that the statement remains actually valid for any skew-shift $\hat{T}$.
\end{rem}

\begin{proof}

If $\varphi$ is any function from $S^1$ into $\R$ and $\mu$ a probability measure on $S^1$, we denote 
$$\left\{\begin{disarray}{l} \|\varphi\|_{l^1}=\sum_{x\in S^1}|\varphi(x)| \\ \|\mu\|_{l^\infty}=\sup_{x\in S^1} \mu(\{x\})\end{disarray}\right.,$$
so that, if $\|\varphi\|_{l^1}<+\infty$:
$$\int_{S^1} \varphi d\mu\leq \|\varphi\|_{l^1}\|\mu\|_{l^\infty},$$
with equality if and only if $\varphi$ is supported on the set $\left\{x\in S^1| \mu(\{x\})=\|\mu\|_{l^\infty}\right\}$.\\

Now, in the context of the statement, let us set
$$E(\omega)=\left\{x\in S^1|\mu_\omega(\{x\})|=\|\mu_\omega\|_{l^\infty}\right\},$$
which is clearly finite and non empty if $\|\mu_\omega\|_{l^\infty}>0$ (which occurs on a set of positive probability by assumption). We are going to prove that these sets $E(\omega)$ satisfy the conclusion of the statement. Let $\varphi:(\omega,x)\mapsto \varphi_\omega(x)$ be the function defined by:
$$\varphi_{\omega}(x)=\left\{\begin{disarray}{ll} \frac{\mathds{1}_{E(\omega)}(x)}{\mathrm{Card}(E(\omega))} &\mbox{ if } \|\mu_\omega\|_{l^\infty}>0 \\ \\
0 &\mbox{ if } \|\mu_\omega\|_{l^\infty}=0\end{disarray}\right.$$
Notice that $\|\varphi_\omega\|_{l^1}=1$ if $\|\mu_\omega\|_{l^\infty}>0$ and $0$ if not. On one hand, we have the equality
\begin{equation}\label{abigay}
\int_{\Omega\times S^1}\varphi d\hat{\mu}=\int_{\Omega}\left(\int_{S^1} \varphi_\omega d\mu_\omega\right)d\mathbb{P}(\omega)=\int_{\Omega}\|\mu_\omega\|_{l^\infty}d\mathbb{P}(\omega),\end{equation}
(using the easy computation $\int_{S^1} \varphi_\omega d\mu_\omega=\|\mu_\omega\|_{l^\infty}$), and on the other hand, we have the chain of inequalities:
\begin{equation}\label{abigay2}\begin{disarray}{ll}\int_{\Omega\times S^1}\varphi\circ \hat{T} d\hat{\mu}&=\int_{\Omega}\left(\int_{S^1} (\varphi_{T\omega}\circ f_\omega) d\mu_\omega\right)d\mathbb{P}(\omega)\\
&\leq  \int_{\Omega}\|\varphi_{T\omega}\circ f_\omega\|_{l^1}\|\mu_\omega\|_{l^\infty}d\mathbb{P}(\omega)\\
&\leq  \int_{\Omega}\|\varphi_{T\omega}\|_{l^1}\|\mu_\omega\|_{l^\infty}d\mathbb{P}(\omega)\\
&\leq\int_{\Omega}\|\mu_\omega\|_{l^\infty}d\mathbb{P}(\omega)\end{disarray}\end{equation}
(using the general equality $\|\psi\circ f\|_{l^1}= \|\psi\|_{l^1}$, valid for $f\in \mathrm{Homeo}(S^1)$, and the fact that $\|\varphi_\omega\|_{l^1}\leq 1$).\\ \\
Combining (\ref{abigay}), (\ref{abigay2}) and he invariance equality $\int \varphi d\hat{\mu}=\int \varphi\circ \hat{T} d\hat{\mu}$, we deduce that the chain of inequalities (\ref{abigay2}) is in fact a chain of equalities. In particular, for $\mathbb{P}$-almost every $\omega$ in $\Omega$,  $\int_{S^1} (\varphi_{T\omega}\circ f_\omega) d\mu_\omega=\|\varphi_{T\omega}\circ f_\omega\|_{l^1}\|\mu_\omega\|_{l^\infty}$, hence $\varphi_{T\omega}\circ f_\omega$ is supported on the set $E(\omega)$. In consequence, for $\mathbb{P}$-almost every $\omega\in\Omega$: 
$$f_\omega^{-1}(E(T\omega))\subset E(\omega).$$ 

 Thus, for $\mathbb{P}$-almost every $\omega$ in $\Omega$, $\mathrm{Card}(E(T\omega))\geq \mathrm{Card}(E(\omega))$, and hence by ergodicity $d=\mathrm{Card}(E(\omega))$ does not depend on $\omega$ (up to a negligible set), and $d<+\infty$ by assumption. In particular, $f_\omega:E(\omega)\rightarrow E(T\omega)$ is a bijection and  for $\mathbb{P}$-almost every $\omega\in\Omega$: 
 $$E(T\omega)=f_\omega(E(\omega)).$$ 

Finally, that last equality means that the set $\mathcal{E}=\bigcup_{\omega\in\Omega}\{\omega\}\times E(\omega)$ is $\hat{T}$-invariant up to a $\hat{\mu}$-negligible set hence using the ergodicity of $\hat{\mu}$ and the fact that $\mathcal{E}$ is not $\hat{\mu}$-negligible by assumption we deduce that in fact $\hat{\mu}(\mathcal{E})=1$, i.e. for $\mathbb{P}$-almost every $\omega$ in $\Omega$, $\mu_\omega(E(\omega))=1$. That means that $\mu_\omega$ is supported on the finite set $E(\omega)$, and by definition all the points of $E(\omega)$ have the same $\mu_\omega$-mass, so
$$\mu_\omega=\frac{1}{\mbox{Card}(E(\omega))}\sum_{x\in E(\omega)}\delta_x=\frac{1}{d}\sum_{x\in E(\omega)}\delta_x,$$
which completes the proof.
\end{proof}

As a consequence, if the probability measures $\mu_{\omega}$ have atoms for a set of $\omega$ of $\mathbb{P}$-positive probability, then Lemma \ref{atoms} implies in particular that for $\mathbb{P}$-almost every $\omega$ in $\Omega$, $\mu_{T\omega}=(f_\omega)_*\mu_\omega$, hence $h(\hat{\mu})=0$, so that the inequality $\lambda_{con}(\hat{\mu})\leq -h(\hat{\mu})$ is trivial.\\

\subsection{Proof of Proposition \ref{ineq-entropy}}
From now on, we assume that $\hat{\mu}$ is ergodic and that the fibred probability measures $\mu_\omega$ have no atoms.\\

The main idea of the proof is to use the Birkhoff theorem to $\log J$ to see that the entropy $h(\hat{\mu})$ represents the exponential rate of decrease of $\frac{d(f_\omega^n)^{-1}_*\mu_{T^n\omega}}{d\mu_\omega}$, and hence of $\frac{\mu_{T^n\omega}(f_\omega^n(I))}{\mu_\omega(I)}$ for $I$ a ``typical'' small interval. However, it is more convenient to work with a slightly modified version of $J$:

\begin{Def}
For $\ep>0$, we define the \emph{approximated Jacobian} $J_\ep=J_\ep(\hat{\mu})$ as
$$J_\epsilon(\omega,x)=\sup\left\{ \frac{\mu_{T\omega}(f_\omega(I))}{\mu_\omega(I)}\,\Big|\,I\ni x,\mu_{\omega}(I)\leq \ep\right\}.$$
and the corresponding \emph{approximated entropy} $h_\ep$ as
$$h_\ep(\hat{\mu})=\begin{cases}
\displaystyle -\int_{\Omega\times S^1}\log J_\ep d\hat{\mu}&\text{ if }\log J_\ep\in L^1(\Omega\times S^1,\hat{\mu})\\
+\infty&\text{otherwise.}\end{cases}$$
\end{Def} 
Notice that $J_\ep$ is well defined thanks to the fact that $\mu_\omega$ has no atoms.\\

In the next lemma, we justify that the definitions of $J_\ep(\hat{\mu})$ and $h_\ep(\hat{\mu})$ are legitimate, in the sense that these quantities are indeed approximations of $J(\hat{\mu})$ and $h(\hat{\mu})$.
\begin{lem}
We have
$$\lim_{\ep\to 0} J_\ep(\hat{\mu})=J(\hat{\mu})$$ 
$\hat{\mu}$-almost everywhere, and
\begin{equation}
\label{entrofatou}
\lim_{\ep\to 0} h_\ep(\hat{\mu})=h(\hat{\mu}).
\end{equation}
\end{lem}
\begin{proof}
The first point is a direct consequence of Proposition \ref{Besicovitch} applied to $\mu=\mu_\omega$ and $\nu=(f_{\omega}^{-1})_*\mu_{T\omega}$.
To prove the second point, we write $\log J_\ep=u_\ep-v_\ep$ with $u_\ep=\sup( \log J_\ep,0)$, $v_\ep=\sup( -\log J_\ep,0)$, and we also write $\log J=u+v$ in the same way. We have that $u_\ep\to u$ and $v_\ep\to v$ $\hat{\mu}$-almost everywhere by the first point. Moreover, using the second part of Proposition \ref{Besicovitch}, we deduce that $\sup_{\ep>0}u_\ep \in L^1(\hat{\mu})$, hence by dominated convergence,
$$\lim_{\ep\to 0} \int_{\Omega\times S^1} u_\ep\, d\hat{\mu}=\int_{\Omega\times S^1} u\, d\hat{\mu}.$$
On the other hand, $v_\ep$ is non negative and increasing as $\ep$ decreases to $0$, hence by Beppo-Levi's Theorem,
$$\lim_{\ep\to 0}\int_{\Omega\times S^1} v_\ep\, d\hat{\mu}=\int_{\Omega\times S^1} v\, d\hat{\mu}.$$
The claim follows.
\end{proof}

The following lemma is the key part of the proof of Proposition \ref{ineq-entropy} (and hence of Theorem \ref{invariance}). It establishes some phenomenon of exponential local contractions under the presence of entropy:
\begin{lem}\label{yoplait}
Let us assume that $h(\hat{\mu})$ is positive. Then, for $\hat{\mu}$-almost every $(\omega,x)\in \Omega\times S^1$, for every $\tilde{h}$ in $(0,h(\hat{\mu}))$, there exists $\delta>0$ such that for any interval $I$ containing 
$x$ such that $\mu_\omega(I)<\delta$,
$$\forall n\in\N,\quad  \mu_{T^n\omega}(f_\omega^n(I))\leq e^{-n\tilde{h}}\mu_\omega(I).$$
\end{lem}
\begin{proof}
Let  $\tilde{h}$ in $(0,h)$ be given. By \eqref{entrofatou} one can choose $\ep>0$ so that $h_\ep(\hat{\mu})>\tilde{h}$. Let us take a \emph{Birkhoff point} $(\omega,x)$ of $\log J_\ep$, that is such that
$$\lim_{n\to +\infty} \frac{1}{n}\sum_{k=0}^{n-1}\log J_\ep\circ \hat{T}^k(\omega,x)=-h_\ep(\hat{\mu}).$$
Note: Birkhoff's Theorem is still valid even when $\log J_\ep\not\in L^1(\hat{\mu})$, because  one can apply Birkhoff Theorem to the function $\sup(\log J_\ep, -M)$ (integrable by Proposition \ref{Besicovitch}) for $M$ arbitrarily large.\\

In particular there exists a constant $C_0=C_0(\omega,x)$ such that 
$$\forall n\in \N,\quad \prod_{k=0}^{n-1} J_\ep\circ \hat{T}^k(\omega,x)\leq C_0e^{-n\tilde{h}}.$$
Let $I$ be an interval containing $x$ small enough so that
$$\mu_\omega(I)\leq \delta:=\frac{\ep}{1+C_0},$$
 and let us set $x_n=f_\omega^n(x)$, $I_n=f_{\omega}^n(I)$. We claim that
\begin{equation}
\label{yapasmieux}\forall n\in \N,\quad \mu_{T^n\omega}(I_n)\leq e^{-n\tilde{h}}\mu_\omega(I).\end{equation}
The proof of the claim is done by induction:
\begin{itemize}
\item For $n=0$, the inequality is trivial.
\item If the inequality is satisfied for $k=0,\ldots, n-1$, then, for $k=0,\ldots, n-1$ the interval $I_k$ contains the point $x_k$ and satisfies $\mu_{T^k\omega}(I_k)\leq \ep$ , hence, by definition of $J_\ep$,
$$\frac{\mu_{T^{k+1}\omega}(I_{k+1})}{\mu_{T^k\omega}(I_k)}=\frac{\mu_{T^{k+1}\omega}(f_{T^k\omega}(I_{k}))}{\mu_{T^k\omega}(I_k)}\leq J_\ep(T^k\omega,x_k)=J_{\ep}\circ \hat{T}^k(\omega,x),$$
and we deduce
$$\mu_{T^n\omega}(I_n)=\mu_\omega(I)\prod_{k=0}^{n-1}\frac{\mu_{T^{k+1}\omega}(I_{k+1})}{\mu_{T^k\omega}(I_k)}\leq \mu_\omega(I)\prod_{k=0}^{n-1}J_\ep\circ \hat{T}^k(\omega,x)\leq C_0e^{-n\tilde{h}}\mu_{\omega}(I).$$
\end{itemize}
Thus, (\ref{yapasmieux}) is true, which completes the proof.
\end{proof}
The phenomenon of local exponential contractions given by Lemma \ref{yoplait} are measured in a ´´$\hat{\mu}$-sense''. It remains to justify that these contractions remain in the standard sense: that is the object of the next lemma, where we prove that $\mu_\omega$ can be replaced by other arbitrary measures.
\begin{lem}\label{yoplait2}
Let $\omega\mapsto \nu_\omega$ be any measurable function from $\Omega$ into the set of probability measures on $S^1$. Then, for $\hat{\mu}$-almost every $(\omega,x)$ in $\Omega\times S^1$, we have:
$$\varlimsup_{y\to x}\varlimsup_{n\to +\infty}\frac{\log(\nu_{T^n\omega}[f_\omega^n(x),f_\omega^n(y)])}{n}\leq -h(\hat{\mu}).$$
\end{lem}
\begin{proof}
The case where $\nu_\omega=\mu_\omega$ is a direct consequence of Lemma \ref{yoplait}, that is
\begin{equation}\label{asdfg}\varlimsup_{y\to x}\varlimsup_{n\to +\infty}\frac{\log(\mu_{T^n\omega}[f_\omega^n(x),f_\omega^n(y)])}{n}\leq -h(\hat{\mu}).\end{equation}
 For the general case, let us set
$$Q^*(w,x)=\sup_{I\ni x} \frac{\nu_\omega(I)}{\mu_\omega(I)}.$$ 
By Proposition \ref{Besicovitch}, $\log^+ Q^* \in L^1(\hat{\mu})$,  hence the Birkhoff's sums $\frac{1}{n}\sum_{k=0}^{n-1} \log^+ Q^*\circ \hat{T}^k$
converge $\hat{\mu}$-almost everywhere, hence in particular $\frac{\log^+ Q^*\circ \hat{T}^k}{n}$ tends to $0$  $\hat{\mu}$-almost everywhere, which implies:
\begin{equation}\label{sdfgh}\varlimsup_{y\to x}\varlimsup_{n\to +\infty}\frac{1}{n}\log\left(\frac{\nu_{T^n\omega}[f_\omega^n(x),f_\omega^n(y)]}{\mu_{T^n\omega}[f_\omega^n(x),f_\omega^n(y)]}\right)\leq \varlimsup_{y\to x}\varlimsup_{n\to +\infty}\frac{1}{n}\log^+ Q^*(\hat{T}^k(\omega,x))=0.\end{equation}
The statement is then a direct consequence of \eqref{asdfg} and \eqref{sdfgh}.
\end{proof}

Using Lemma \ref{yoplait2} with $\nu_\omega$ the Lebesgue measure, we obtain that $\lambda_{con}(\hat{\mu})\leq -h(\hat{\mu})$. That completes the proof of Proposition \ref{ineq-entropy}, and hence of Theorem \ref{invariance}

\subsection{Exponent of contraction in RDS}
We go back to the context of random walks on $\mathrm{Homeo}(S^1)$. In this particular case, Theorem \ref{invariance} becomes:
\begin{cor}\label{invarianceiid}~
	Let $(G,\nu)$ a random dynamical system on $S^1$, and let $\mu$ be a stationary probability of the system.
	Then
	\begin{itemize}
		\item either $\lambda_{con}(\mathbb{P}\times\mu)<0$  (contraction),
		\item or $f_*\mu=\mu$ for $\nu$-almost every homeomorphism $f$ (and so for any $f$ in $G$)  (invariance).
	\end{itemize}
\end{cor}

Thus,we obtain information at typical points $x\in S^1$ for the stationary probability measures of the systems. But it is actually possible to deduce information at every point $x$ of the circle. To do this, we are going to use the following general fact of random dynamical systems:
\begin{prop}\label{cluster}
Let $(G,\nu)$ be a RDS on a compact metric space $(X,d)$, $(\Omega,\mathbb{P})=(G^\N,\nu^\N)$ the associated probabilty space, and let $x_0$ be a point of $X$. Then, for
$\mathbb{P}$-almost every $\omega$, the set $\Pi_{\omega,x_0}$ of weak-$*$ cluster values of the sequence of probability measures $\left(\frac{1}{N}\sum_{n=0}^{N-1}\delta_{f_{\omega}^n(x_0)}\right)_{N\in \N}$ is constituted of stationary probability measures of the RDS.
\end{prop}
This proposition is the analogue of the standard  Krylov-Bogolyubov Theorem for RDS. The proof can be found in \cite{Deroin-cours} (French), or it can be seen as a consequence of Lemma 2.5 of \cite{Furstenberg}. We are going to use it to extract punctual informations on $\lambda_{con}$ from the informations on stationary measures:

\begin{prop}\label{contractuel}
Let $\omega\mapsto (f_\omega^n)_{n\in\N}$ a random walk on $\mathrm{Homeo}(S^1)$ and let $x_0$ in $S^1$. Then for
$\mathbb{P}$-almost every $\omega$ we have $$\lambda_{con}(\omega,x_0)\leq \inf_{\mu\in \Pi_{\omega,x_0}}\lambda_{con}(\mathbb{P}\otimes \mu),$$
where $\Pi_{\omega,x_0}$ is defined as in Proposition \ref{cluster}.
\end{prop}

The proof of Proposition \ref{contractuel} begins by noticing two elementary facts on the function $(\omega,x)\mapsto\lambda_{con}(\omega,x)$.
\begin{lem}\label{elem}
The function $\lambda_{con}$ is $\hat{T}$-invariant ($\lambda_{con}\circ \hat{T}=\lambda_{con}$), and for any $\omega$ in $\Omega$, the function $x\mapsto \lambda_{con}(\omega,x)$ is upper semicontinuous.
\end{lem}
\begin{proof}
The invariance property $\lambda_{con}\circ \hat{T}= \lambda_{con}$ comes from the fact that an interval $I$ containing $f_0(x)$ is contracted by the sequence $(f_{T\omega}^n)$ if and only if $f_0^{-1}(I)$ is an interval containing $x$ contracted by $(f_\omega^n)$.

The upper semicontinuity of $\lambda_{con}$ comes from the fact that if $\lambda_{con}(\omega,x)<c$, then there exists an interval $I$ containing $x$ such that $\mathrm{diam}(f_\omega^n(I))=O(e^{-nc})$ and hence $\lambda_{con}(\omega,\cdot)<c$ on $I$.
\end{proof}

Then, Proposition \ref{contractuel} is actually a direct consequence of a much more general fact of random dynamical systems:

\begin{lem}\label{randomfunction}
Let $(G,\nu)$ be a RDS on a compact space $X$ and let $(\Omega,\mathbb{P})$ and $\hat{T}$ be defined as in Definition \ref{randomwalk}. Let $\varphi:\Omega\times X\mapsto \R$  be a measurable positive function such that:
\begin{itemize}
\item for every $\omega\in \Omega$, $x\mapsto \varphi(\omega,x)$ is lower semicontinuous,
\item $\varphi\circ \hat{T}\leq \varphi$ on $\Omega\times X$.
\end{itemize}
Finally, let $x_0$ be a point of $X$ and for $\omega$ in $\Omega=G^\N$, let $\Pi_{\omega,x_0}$ be defined as in Proposition \ref{cluster}. Then, for $\mathbb{P}$-almost every $\omega$ in $\Omega$,
$$\varphi(\omega,x_0)\geq \sup_{\mu \in \Pi_{\omega,x_0}} \int_{\Omega\times X} \varphi d(\mathbb{P}\otimes \mu).$$
\end{lem}
\begin{proof}
Let $\mathcal{F}_n$ the $\sigma$-algebra generated by the $n$ first canonical projections $p_k:\omega=(f_j)_{j\geq 0}\mapsto f_k$, and set
$$\bar{\varphi}(x)=\E[\varphi(\cdot,x)],$$
$$\Lambda_n=\E[\varphi(\cdot,x_0) | \mathcal{F}_n]$$
Levy's zero-one law says that $\Lambda_n\xrightarrow[n\to +\infty]{} \varphi(\cdot,x_0)$ almost surely. On the other hand, from the inequality 
$$\varphi(\omega,x_0)\geq \varphi\circ \hat{T}^n(\omega,x_0)=\varphi(T^n\omega, f_\omega^n(x_0)),$$
 we deduce by taking the conditional expectation with respect to $\mathcal{F}_n$ that for $\mathbb{P}$-almost every $\omega$, 
$$\Lambda_n(\omega)\geq\bar{\varphi}(f_\omega^n(x_0)).$$
Hence, using the Cesaro theorem, for $\mathbb{P}$-almost every $\omega$, 

\begin{equation}\label{lambda}\varphi(\omega,x_0)=\lim_{N\to +\infty} \frac{1}{N}\sum_{n=0}^{N-1}\Lambda_n(\omega)\geq \varlimsup_{N\to +\infty} \frac{1}{N}\sum_{n=0}^{N-1}\bar{\varphi}(f_\omega^n(x_0))\end{equation}
Now, we know that $\bar{\varphi}$ is lower semicontinuous thanks to the lower semicontinuity of $\varphi(\omega,\cdot)$ and Fatou's lemma: indeed, for any $x$ in $X$,
\begin{equation}\label{fatou}\varliminf_{y\to x} \bar{\varphi}(y)=\varliminf_{y\to x} \E[\varphi(\cdot,y)]\geq
\E[\varliminf_{y\to x} \varphi(\cdot,y)]\geq \bar{\varphi}(x).\end{equation}
As a consequence, we can write:
$$\bar{\varphi}=\inf\{\psi:X\rightarrow \R \mbox{ continuous }| \psi\leq \bar{\varphi}\},$$
and for every such continuous function $\psi\leq \bar{\varphi}$, we have by (\ref{lambda}) and definition of $\Pi_{\omega,x_0}$:
$$\varphi(\omega,x_0)\geq\varlimsup_{N\to +\infty} \frac{1}{N}\sum_{n=0}^{N-1}\psi(f_\omega^n(x_0))=\sup_{\mu\in \Pi_{\omega,x_0}}\int_X \psi d\mu$$
Since $\psi$ is arbitrary, we deduce that
$$\varphi(\omega,x_0)\geq \sup_{\mu\in \Pi_{\omega,x_0}}\int_X \bar{\varphi}\, d\mu=\sup_{\mu\in \Pi_{\omega,x_0}}\int_{\Omega\times X}\varphi\, d\mathbb{P}d\mu.$$
\end{proof}

Let us conclude by deducing the following corollary, which is only a reformulation of Theorem \ref{main} in terms of exponent of contraction:
\begin{cor}
Let $\omega\mapsto (f_\omega^n)_{n\in\N}$ a random walk generated by a probability measure $\nu$ on $\mathrm{Homeo}(S^1)$, and let us assume that there is no probability measure on $S^1$ invariant by $\nu$-almost every homeomorphism. Then there exists $\lambda_0<0$ such that for any $x$ in $S^1$, for  $\mathbb{P}$-almost every $\omega$ in $\Omega$,
$$\lambda_{con}(\omega,x)\leq \lambda_0$$
\end{cor}
\begin{proof}
By Corollary \ref{invarianceiid}, $\lambda_{con}(\mathbb{P}\otimes \mu)<0$ for any stationary probability measure $\mu$, hence Proposition \ref{contractuel} applied to $-\lambda_{con}$ immediately implies that for any $x$ in $S^1$, $\omega\mapsto \lambda_{con}(\omega,x)$ is negative $\mathbb{P}$-almost everywhere. To obtain a uniform negative bound, let us notice that this negativity implies the negativity of $\overline{\lambda}_{con}(x)=\int_{\Omega}\lambda_{con}(\omega,x)d\mathbb{P}(\omega)$. Thus, $x\mapsto \overline{\lambda}_{con}(x)$ is pointwise negative, and is also upper-semicontinuous by Fatou's lemma as in the computation (\ref{fatou}), hence is uniformly bounded from above by some negative number $\lambda_0$. Then, using Proposition \ref{contractuel} one more time, we obtain that for any $x$ in $S^1$ and $\mathbb{P}$-almost every $\omega$:
$$\lambda_{con}(\omega,x)\leq \inf_{\mu\in \Pi_{\omega,x}}\lambda_{con}(\mathbb{P}\otimes \mu)=\inf_{\mu\in \Pi_{\omega,x}}\int_{S^1}\overline{\lambda}_{con}d\mu\leq \lambda_0.$$
That achieves the proof of the corollary, and hence of Theorem \ref{main}.
\end{proof}

\section{Locally contracting random dynamical systems}\label{sec-consequences}
In this section, we are going to study the properties of a general random walk on a compact metric space. This section can be read independantly of the remainder of the paper, except that we are going to use Lemma \ref{randomfunction} proved in the previous section, and that we will use the notations given in Definition \ref{randomwalk}. Thus, throughout the whole section:
\begin{itemize}
	\item $(G,\nu)$ is a random dynamical system on a compact metric space $(X,d)$, that is to say that $G$ a semigroup of continuous transformations of $X$ and $\nu$ a probability measure on $G$.
	\item $(\Omega,\mathbb{P})=(G^\N,\nu^\N)$ is the associated probability space, $\omega\mapsto (f_\omega^n)$ the associated random walk, defined by
	$$f_\omega^n=f_{n-1}\circ\cdots\circ f_0$$
    (with the implicit notation $\omega=(f_n)_{n\in\N}$), and $\hat{T}:(\omega,x)\mapsto (T\omega,f_0(x))$ the associated skew-shift on $\Omega\times X$.
\end{itemize}
We are going to study the properties of such RDS satisfying the \textit{property of local contractions}:
\begin{ass}\label{contracting}
For every  $x$ in $X$, for $\mathbb{P}$-almost every $\omega$ in $\Omega$, there exists a neighbourhood $B$ of $x$ such that 
$$\mathrm{diam}(f_\omega^n(B))\xrightarrow[n\to+\infty]{} 0.$$
\end{ass}
\begin{rem}
	By Theorem \ref{main}, Assumption \ref{contracting} is satisfied when $G$ is a subgroup of $\mathrm{Homeo}_+(S^1)$ without invariant probability measure. It is also satisfied if $X$ is a manifold, $G$ a semigroup of diffeomorphisms of $X$ and such that all the Lyapunov exponents of the random walk are negative.
\end{rem}

\subsection{Preliminaries on random sets}
In this part, we state some general results on the RDS, concerning the structure of the sets invariant by $\hat{T}$. We do not use Assumption \ref{contracting} in this part.

\begin{prop}\label{randomset}
Let $\mathcal{E}=\cup_{\omega\in \Omega}\{\omega\}\times U(\omega)$ a subset of $\Omega\times X$ backward-invariant by $\hat{T}$ (i.e.~$\hat{T}^{-1}(\mathcal{E})\subset \mathcal{E}$) such that $U(\omega)$ is open in $X$ for every $\omega$ in $\Omega$. Let us assume that $$(\mathbb{P}\otimes \mu)(\mathcal{E})>0$$
 for every stationary ergodic probability measure $\mu$. Then actually,
$$(\mathbb{P}\otimes \mu)(\mathcal{E})=1$$
for every probability measure $\mu$ on $X$ (not necessarily stationary).
\end{prop}
\begin{proof}
Firstly, the set of the stationary probability measures is the convex hull of the the set of the ergodic ones, so that the inequality $(\mathbb{P}\otimes \mu)(\mathcal{E})>0$ remains valid for any $\mu$ stationary. Then, by applying Lemma \ref{randomfunction} to $\varphi=\mathds{1}_{\mathcal{E}}$, for any $x_{0}$ in $X$ and for almost every $\omega$ in $\Omega$, we have with the notations of the lemma:
$$\mathds{1}_{\mathcal{E}}(\omega,x_0)\geq \sup_{\mu \in \Pi_{\omega,x_0}} (\mathbb{P}\times \mu)(\mathcal{E})>0,$$
hence $(\omega,x_0)\in \mathcal{E}$. The result follows.
\end{proof}

The second proposition shows that the fibres of a $\hat{T}$-invariant set cannot have many connected components (that will be the main ingredient for the proof of Theorem \ref{Lejan-Antonov}).
\begin{prop}\label{randomcon}
 Let $\mathcal{E}=\cup_{\omega\in \Omega}\{\omega\}\times E(\omega)$ a subset of $\Omega\times X$ totally invariant by $\hat{T}$ ($\hat{T}^{-1}(\mathcal{E})=\mathcal{E}$). Then, for every stationary ergodic probability measure $\mu$, for $\mathbb{P}$-almost every $\omega$ in $\Omega$, $E(\omega)$ has only a constant finite number $d$ of connected components of $\mu$-measure positive, and all of them have same measure $\frac{1}{d}$.
\end{prop}
\begin{proof}
In order to prove this proposition, we extend (canonically) the skew-shift $\hat{T}$ on $G^\Z\times X$, in an invertible context, allowing to look at the ''past'' of the RDS. This procedure is standard, we resume in the following lemma the properties of the extension we use (we refer to \cite{Lejan} for the details).
\begin{lem}
Let $\tilde{\Omega}=G^\Z$ and $\tilde{\mathbb{P}}=\nu^{\Z}$ . The transformation $\hat{T}:(\omega,x)\mapsto (T\omega, f_0(x))$ admits an invariant ergodic probability measure $\hat{\mu}$  on $\tilde{\Omega}\times X$ of the form $d{\hat{\mu}}=d\mu_\omega(x)d\mathbb{P}(\omega)$, with:
\begin{itemize}
\item  the function $\omega\mapsto \mu_\omega$ depending only on the negative coordinates of $\omega$,
\item  $\int_{\tilde{\Omega}} \mu_\omega d\tilde{\mathbb{P}}(\omega)=\mu$,
\item for $\tilde{\mathbb{P}}$-almost every $\omega$ in $\tilde{\Omega}$, $\mu_{T\omega}=(f_0)_*\mu_\omega$.
\end{itemize}
\end{lem}

This process will allow us to prove the following general lemma:
\begin{lem}\label{ergodic}
Let $(E(\omega,x))_{(\omega,x)\in\Omega\times X}$ be a family of Borelian subsets of $X$ such that 
$$\forall (\omega,x)\in \Omega\times X,\quad E(\hat{T}(\omega,x))=f_0(E(\omega,x)).$$
Then the function $(\omega,x)\mapsto \mu(E(\omega,x))$ is constant $(\mathbb{P}\otimes \mu)$-almost everywhere.
\end{lem}

\begin{proof}
Let us extend canonically $(\omega,x)\mapsto E(\omega,x)$ to $\tilde{\Omega}\times X$ (by setting $E((f_k)_{k\in\Z},x):=E((f_k)_{k\in\N},x)$). For every $(\omega,x)\in \tilde{\Omega}\times X$, $E(\omega,x)=(f_0)^{-1}(E(\hat{T}(\omega,x)))$, hence
$$\mu_\omega(E(\omega,x))=(f_0)_*\mu_\omega (E(\hat{T}(\omega,x)))=\mu_{T\omega} (E(\hat{T}(\omega,x))).$$
The function $(\omega,x)\mapsto \mu_\omega(E(\omega,x))$ is hence $\hat{T}$-invariant on $\tilde{\Omega}\times X$.
By ergodicity of $\hat{\mu}$, there exists a constant $c$ such that for $\hat{\mu}$-almost every $(\omega,x)$ in $\tilde{\Omega}\times X$, $\mu_\omega(E(\omega,x))=c$. Since $\mu_{\omega}$ only depends  on the negative coordinates of $\omega$ and $E(\omega,x)$ only depends  on the non negative coordinates of $\omega$, we deduce by integration of this equality over the negative coordinates of $\omega$ that for $(\mathbb{P}\otimes \mu)$-almost every $(\omega,x)$ in $\Omega\times X$, $\mu(E(\omega,x))=c$.
\end{proof}
Proposition \ref{randomcon} follows by choosing $E(\omega,x)$ to be the connected component of $x$ in $E(\omega)$ (with the convention $E(\omega,x)=\emptyset$ if $x\notin E(\omega)$), satisfying the relation $E(\hat{T}(\omega,x))=f_0(E(\omega,x))$. For any ergodic probability measure $\mu$ of the RDS, by Lemma \ref{ergodic}, for $\mathbb{P}$-almost every $\omega$, the function $x\mapsto \mu(E(\omega,x))$ is equal to some positive constant $c$ $\mu$-almost everywhere, which means that all the connected components of $U(\omega)$ which are not $\mu$-negligible have the same $\mu$-measure $c$. In particular there is only a finite number of them, namely $\frac{1}{c}$.
\end{proof}

\subsection{Stationary trajectories}
We prove in this part that the property of local contractions implies that the number of ergodic stationary probability measures is finite, and the trajectory of every point almost surely distributes with respect of one of them.\\

\begin{Def}\label{contracdef}
We say that a ball $B$ is contractible if there exists a set $\Omega '\subset \Omega$ of $\mathbb{P}$-positive probability such that, for $\omega$ in $\Omega '$, $\mathrm{diam}(f_\omega^n(B))\xrightarrow[n\to +\infty]{}0$.	
\end{Def}
Assumption \ref{contracting} implies that every point contains a contractible neighbourhood.
\begin{lem}\label{contractible}
If $B\subset X$ is a contractible ball, then there exists at most one ergodic stationary probability measure $\mu$ such that $\mu(B)>0$.
\end{lem}
\begin{proof}
Let $\mu_1$ and $\mu_2$ be two ergodic stationary measures such that $\mu_1(B)\not=0$ and $\mu_2(B)\not=0$. By Birkhoff's theorem one can find $x$ and $y$ in $B$ such that for $\mathbb{P}$-almost every $\omega$ in $\Omega$ , for every continuous $\varphi:X\rightarrow \R$,
\begin{equation}\label{limit123}\begin{disarray}{l}\frac{1}{N}\sum_{n=0}^{N-1}\varphi(f_\omega^n(x))\xrightarrow[N\to+\infty]~\int_{X}\varphi d\mu_1\\
\frac{1}{N}\sum_{n=0}^{N-1}\varphi(f_\omega^n(y))\xrightarrow[N\to+\infty]~\int_{X}\varphi d\mu_2.\end{disarray}\end{equation}
Since $B$ is contractible, one can choose such an $\omega$ for which $\mathrm{diam}(f_\omega^n(B))$ tends to $0$ as $n$ tends to $+\infty$. Then, for every continuous mapping $\varphi:X\rightarrow \R$, $\varphi(f_\omega^n(x))-\varphi(f_\omega^n(y))$ tends to $0$ as $n$ tends to $+\infty$, hence we conclude from (\ref{limit123}) that $$\int_{X}\varphi d \mu_1=\int_{X}\varphi d \mu_2,$$
so that $\mu_1=\mu_2$.
\end{proof}

\begin{prop}\label{finite}
 If the RDS $(G,\nu)$ satisfies Assumption \ref{contracting}, then it has a finite number $d$ of ergodic stationary probability measures $\{\mu_1,\ldots,\mu_d\}$. Their respective topological supports $F_1,\ldots,F_d$ are pairwise disjoints, and are exactly the minimal invariant compacts of $G$.
\end{prop}
\begin{proof}
Each point of $x$ is the centre of a contractible ball, hence by compactness, we can cover $X$ by a finite number of contractible balls $B_1,\ldots,B_d$. By  Lemma \ref{contractible}, for each $i$, there is at most one ergodic probability measure $\mu$ such that $\mu(B_i)\not=0$. Hence, there are at most $d$ stationary ergodic probability measures.\\

Let $\{\mu_1,\ldots,\mu_d\}$ be the set of the ergodic probability measures and let $F_i=\mathrm{supp}(\mu_i)$ be the topological support of $\mu_i$. If $x\in F_i\cap F_j$, then if $B$ a contractible ball centred at $x$, we have $\mu_i(B)\not=0$ and $\mu_j(B)\not=0$, hence by Lemma \ref{contractible}, $\mu_i=\mu_j$. The sets $F_1,\ldots,F_d$ are hence pairwise disjoint.\\

If $F$ is a minimal closed invariant subset of $X$, then there exists a stationary ergodic probability measure $\mu_i$ supported in $F$. And since $F_i=\mathrm{supp}(\mu_i)$ is invariant by $G$, we have $F=F_i$ by minimality of $F$.

Conversely, let $i$ be in $\{1,\ldots,d\}$. The closed set $F_i=\mathrm{supp}(\mu_i)$ is invariant by $G$, hence it contains a minimal invariant closed subset $F$. By the previous point, $F=F_j$ for some $j$, but since the $F_1,\ldots,F_d$ are pairwise disjoint, necessarily $i=j$ and hence $F_i=F$ is a minimal invariant subset.\\
\end{proof}

\begin{prop}\label{adhe}
Let us assume that the RDS $(G,\nu)$ satisfies Assumption \ref{contracting} and let $\mu_1,\ldots,\mu_d$ and $F_1,\ldots,F_d$ be as in Proposition \ref{finite}. Then for every $x$ in $X$, for $\mathbb{P}$-almost every $\omega$ in $\Omega$, there exists a (unique) integer $i=i(\omega,x)$ in $\{1,\ldots,d\}$ such that:
\begin{itemize}
\item The set of cluster values of the sequence $(f_\omega^n(x))_{n\geq 0}$ is exactly $F_i$.

\item The sequence of probability measures $\frac{1}{N}\sum_{n=0}^{N-1}\delta_{f_\omega^n(x)}$ weakly-$*$ converges to $\mu_i$ in $C(X,\R)^*$.
\end{itemize}
\end{prop}
\begin{proof}
Let us consider $\mathcal{E}_0$ to be the set of the points $(\omega,x)$ such that there exists a neighbourhood of $x$ contracted by $(f_\omega^n)_n$, and
let
$$\mathcal{E}_i=\left\{(\omega,x)\in \mathcal{E}_0 \left| \frac{1}{N}\sum_{n=0}^{N-1}\delta_{f_\omega^n(x)}\xrightarrow[n\to+\infty]{C(X,\R)^*} \mu_i \right.\right\}=\bigcup_{\omega\in\Omega}\{\omega\}\times U_i(\omega)$$
and 
$$\tilde{\mathcal{E}}_i=\left\{(\omega,x)\in \mathcal{E}_0 \left| \mbox{Acc}\left\{f_\omega^n(x),n\in\N\right\}=F_i \right.\right\}=\bigcup_{\omega\in\Omega}\{\omega\}\times \tilde{U}_i(\omega).$$
Then:
\begin{itemize}
	\item $\mathcal{E}_i$ and $\tilde{\mathcal{E}}_i$ are totally invariant by $\hat{T}$.
	\item If $\omega$ belongs to $\omega$ and $B$ is a ball such that $\mbox{diam}(f_\omega^n(B))$ tends to $0$ when $n$ tends to $+\infty$, then either $B\cap U_i(\omega)=\emptyset$ (resp $B\cap \tilde{U}_i(\omega)=\emptyset$) or $B\subset U_i(\omega)$ (resp $B\subset \tilde{U}_i(\omega)$). In consequence, $U_i(\omega)$ and $\tilde{U}_i(\omega)$ are open.
	\item By Birkhoff Theorem, for $(\mathbb{P}\otimes \mu_i$)-almost every $(\omega,x)$ in $\Omega\times X$,
	$$\frac{1}{N}\sum_{n=0}^{N-1}\delta_{f_\omega^n(x)}\xrightarrow[n\to+\infty]{C(X,\R)^*} \mu_i.$$
	In particular, $\mbox{Acc}\left\{f_\omega^n(x),n\in\N\right\}\supset F_i$, and if $x$ belongs to $F_i$, $\mbox{Acc}\left\{f_\omega^n(x),n\in\N\right\}\subset F_i$ by invariance of $F_i$. Thus, $(\mathbb{P}\otimes \mu_i)(\mathcal{E}_i)=(\mathbb{P}\otimes \mu_i)(\tilde{\mathcal{E}}_i)=1$.
\end{itemize}
In consequence, one can apply Proposition \ref{randomset} to the sets $\mathcal{E}=\cup_i \mathcal{E}_i$ and $\tilde{\mathcal{E}}=\cup_i \tilde{\mathcal{E}}_i$ and get:
$$\forall x\in X, \quad (\mathbb{P}\times \delta_x)(\mathcal{E})=(\mathbb{P}\times \delta_x)(\tilde{\mathcal{E}})=1.$$
The claimed result follows.
\end{proof}

\subsection{Dynamics of the transfer operator}
We study in this part the sequence of the iterates of the transfer operator $P$ of a RDS applied to a continuous test function $\varphi$. We prove that under the property of local contractions, this sequence $(P^n\varphi)_{n\in\N}$ always converges uniformly in the Cesaro sense to a harmonic function, and that it actually converges uniformly in the standard sense if the RDS is aperiodic (in the sense of Definition \ref{indeco}).\\

The transfer operator $P$ of the system is defined on measurable bounded functions $\varphi:X\rightarrow \R$, by
$$P\varphi=\int_{G} \varphi\circ f d\nu(f).$$ 
The iterates of $P$ are given by 
$$P^n\varphi=\int_\Omega\varphi\circ f_\omega^n\,d\mathbb{P}(\omega),$$
so that the dynamics of $P$ represents the evolution of the law of the random variables $\omega\mapsto f_\omega^n(x)$.

\begin{lem}\label{equicontinuous}
If the RDS $(G,\nu)$ satisfies Assumption \ref{contracting}, then for every continuous $\varphi:X\rightarrow \R$, the family $(P^n\varphi)_{n\in\N}$ is equicontinuous on $X$.
\end{lem}
\begin{proof}
Let $\ep>0$, and let $\delta>0$ be such that
 $$\forall x,y \in X^2, d(x,y)\leq\delta\Rightarrow |\varphi(x)-\varphi(y)|\leq \ep.$$
   Let $x$ be in $X$. Thanks to Assumption \ref{contracting}, we can find a ball $B$ centred at $x$ and a subset $\Omega '\subset \Omega$ of probability more than $1-\ep$ such that:
$$ \forall n\in \N, \forall \omega\in \Omega',\quad \mathrm{diam}( f_\omega^n(B))\leq \delta.$$
We deduce that for every integer $n$ and every $y$ in $B$:
$$\begin{disarray}{ll}\displaystyle|P^n\varphi(x)-P^n\varphi(y)|&\leq\int_{\Omega}|\varphi(f_\omega^n(x))-\varphi(f_\omega^n(y))|d\mathbb{P}(\omega)
\\&\leq \ep\mathbb{P}(\Omega')+2\|\varphi\|_{\infty}\mathbb{P}(\Omega-\Omega')\\
&\leq (1+2\|\varphi\|_\infty)\ep.\end{disarray}
$$
Thus, $(P^n\varphi)_{n\in\N}$ is equicontinuous at $x$. Since $x$ is arbitrary and $X$ is compact, $(P^n\varphi)_{n\in\N}$ is equicontinuous on $X$.
\end{proof}

\begin{prop}\label{Furstenberg}
We assume that the RDS $(G,\nu)$ satisfies Assumption \ref{contracting}, and we keep the notations of Proposition \ref{finite}, i.e. $\mu_1,\ldots,\mu_d$ and $F_1,\ldots,F_d$ are respectively the ergodic stationary probability measures of the RDS and their topological supports. Then:
\begin{itemize}
\item The vector space $E_0=\{\varphi\in C(X,\R)\mid P\varphi=\varphi\}$ of the harmonic continuous functions of the RDS has finite dimension $d$, and one can find a basis $(u_1,\ldots,u_d)$ of $E_0$ such that $u_i$ is valued in $[0,1]$, $u_i=\delta_{i,j}$ on $F_j$ and $\sum_i u_i=1$ on $X$.

\item For every continuous $\varphi:X\rightarrow \R$, we have 
$$ \frac{1}{N}\sum_{n=0}^{N-1} P^n\varphi\xrightarrow[N\to +\infty]{\|\cdot\|_\infty} \psi$$ 
where $\psi$ is the element of $E_0$ given by 
$$\psi(x)=\sum_{i=1}^d \left(\int_X\varphi d\mu_i\right)u_i(x).$$
\end{itemize}
\end{prop}

\begin{proof}

Let $\varphi:X\rightarrow \R$ be a continuous function, and let $x$ be in $X$. With $i(\omega,x)$ defined as in Proposition \ref{adhe}, we have for $\mathbb{P}$-almost every $\omega$ in $\Omega$:
$$\frac{1}{N}\sum_{n=0}^{N-1}\varphi(f_\omega^n(x))\xrightarrow[n\to+\infty]{} \int_X \varphi d\mu_{i(\omega,x)}.$$
Integrating in $\omega$, we deduce by dominated convergence that
\begin{equation}
\label{mkjh}
\frac{1}{N}\sum_{n=0}^{N-1}P^n\varphi(x)\xrightarrow[n\to+\infty]{} \sum_{i=1}^d u_i(x)\int_X \varphi d\mu_{i},\end{equation}
where $u_i(x)=\mathbb{P}(\omega\in \Omega \mid i(\omega,x)=i)$. Since the sequence $\left(\frac{1}{N}\sum_{n=0}^{N-1}P^n\varphi\right)_{n\in\N}$ is equicontinuous by Lemma \ref{equicontinuous}, the convergence (\ref{mkjh}) is in fact uniform in $x$.

The only non trivial property to prove on the functions $u_i$ is their continuity. For a given $i$ , we choose $\varphi$ continuous such that $\varphi=\delta_{i,j}$ on $K_j$, so that (\ref{mkjh}) becomes
$$u_i=\lim_{N\to +\infty}\frac{1}{N}\sum_{n=0}^{N-1}P^n\varphi$$
where the limit is uniform. The continuity of $u_i$ follows. 
\end{proof}

We will strenghten the result in the case of aperiodic systems. Let us recall the definition of aperiodicity given in Section \ref{results} in the case of the circle:

\begin{Def}\label{aper}
	The RDS $(G,\nu)$ on $X$ (resp. the random walk $\omega\mapsto (f_\omega^n)_{n\in\N}$) is said to be aperiodic if there does not exist a finite number $p\geq 2$ of pairwise disjoints closed subsets $F_1,\ldots,F_p$ of $X$ such that for $\nu$-almost every homeomorphism $g$, $g(F_i)\subset F_{i+1}$ for $i=1,\ldots,p-1$ and $g(F_p)\subset F_1$.
\end{Def}
\begin{rem}
	As already noticed in the particular case of the circle in Section \ref{results}, if a random walk $\omega\mapsto (f_\omega^n)_{n\in\N}$ acts minimally on $X$ and if $X$ is connected, then it is automatically aperiodic.
\end{rem}
We can state our result, which studies the convergence of th sequence $(P^n)_{n\in\N}$:
\begin{prop}\label{transfert}
We assume that the RDS $(G,\nu)$ satisfies Assumption \ref{contracting}, and we assume also that it is aperiodic. Then, keeping the notations of Proposition \ref{Furstenberg}, we have actually
$$ P^n\varphi\xrightarrow[n\to +\infty]{\|\cdot\|_\infty} \psi=\sum_{i=1}^d \left(\int_X\varphi d\mu_i\right)u_i$$ 
\end{prop}
The aperiodicity of the system is used to obtain the following fact, whose proof is postpone:
\begin{lem}\label{indemin}
If $\omega\mapsto (f_\omega^n)_{n\geq 0}$ acts minimally on $X$ and is aperiodic, then for any positive integer $p$, $\omega\mapsto (f_\omega^{pn})_{n\geq 0}$ also acts minimally on $X$.
\end{lem}
\begin{proof}[Proof of Proposition \ref{transfert}]
Let $\varphi:X\rightarrow \R$ be a continuous mapping. Thanks to Lemma \ref{equicontinuous}, the only thing we need to prove is that $(P^n\varphi)_{n\in\N}$ has only one cluster value in $C(X,\R)$, namely $\sum_i \left(\int \varphi d\mu_i\right)u_i$. Thus, let $\psi=\lim_{k\to +\infty} P^{n_k}\varphi$ be a cluster value of $(P^n\varphi)_n$.\\

Firstly, up to to extracting the candidate limit $\sum_i \left(\int \varphi d\mu_i\right)u_i$ to $\varphi$, we can assume that $\int_{S^1}\varphi d\mu_i=0$ for $i=1,\ldots d$, so that we want to prove that $\psi=0$.

Secondly, we can reduce the problem to the case where $\varphi=\psi$: indeed, up to extracting a subsequence, we can assume that $m_k=n_{k+1}-n_k$ tends to $+\infty$ when $k$ tends to $+\infty$. Using that $P$ is contracting for $\|\cdot\|_\infty$, we have
\begin{equation}
\label{constant2}
\|P^{m_k}\psi-\psi\|_\infty\leq \|P^{m_k}(\psi-P^{n_{k}}\varphi)\|_\infty
+\|P^{n_{k+1}}\varphi-\psi\|_\infty\xrightarrow[k\to+\infty]{}0 \end{equation}\\

Thus, from now on we assume that:
\begin{itemize}
\item $\displaystyle \int_{S^1}\varphi d\mu_i=0 \mbox{ for } i=1,\ldots d,$
\item $\displaystyle P^{n_k}\varphi\xrightarrow[k\to +\infty]{\|\cdot\|_\infty} \varphi,$
\end{itemize}
and we want to prove that $\varphi=0$.
We begin by treating the restriction of the problem to a minimal subset $F_i=\mathrm{supp}(\mu_i)$. We will use the following remark:
\begin{lem}
For any continuous $\varphi:X\rightarrow \R$ and any positive integer $k$, we have $\|P^k\varphi\|_{L^2(\mu_i)}\leq \|\varphi\|_{L^2(\mu_i)}$, with equality if and only if for $\mathbb{P}$-almost every $\omega$, $\omega'$, $\varphi\circ f_\omega^k=\varphi\circ f_{\omega'}^k$ on $F_i$.
\end{lem}
\begin{proof}The inequality is just a consequence of the Jensen inequality $P^k(\varphi)^2\leq P^k(\varphi^2)$ and of the $P^k$-invariance of $\mu$, and in the equality case of the Jensen inequality,we have that for almost every $\omega$, $\omega'$, $\varphi\circ f_\omega^k=\varphi\circ f_{\omega'}^k$ $\mu_i$-almost everywhere, hence on $F_i$ by continuity.\end{proof}

 By the lemma, the sequence $(\|P^n\varphi\|_{L^2(\mu_i)})$ is non increasing. For any integer $p$, writing that
$$\|P^{n_k}\varphi\|_{L^2(\mu_i)}\leq \|P^{n_k+p}\varphi\|_{L^2(\mu_i)}\leq \|P^{n_{k+p}}\varphi\|_{L^2(\mu_i)},$$
and passing to the limit, we obtain that $\|P^p\varphi\|_{L^2(\mu_i)}=\|\varphi\|_{L^2(\mu_i)}$,
and hence by the lemma, for $\mathbb{P}$-almost every $\omega$, $\varphi\circ f_\omega^p=P^p\varphi$ on $F_i$.
 
As a consequence, we obtain that for $\mathbb{P}$-almost every $\omega$ in $\Omega$
$$\varphi\circ f_\omega^{n_k}=P^{n_k}\varphi\xrightarrow[k\to +\infty]{\|\cdot\|_\infty} \varphi \text{ on } F_i$$
In particular, if $B$ is a contractible ball of $F_i$, then $\varphi$ is constant on $B$. By compactness, $\varphi$ only takes a finite number of values on $F_i$. We deduce that fixing an integer $p=n_k$ with $k$ large enough, we have $\varphi\circ f_\omega^p=\varphi$ on $F_i$ for $\mathbb{P}$-almost every $\omega$. Hence $\varphi$ is constant on $F_i$ by Lemma \ref{indemin}, and this constant is necessarily $\int \varphi d\mu_i=0$.

\medskip

We now go back to the whole space: we know that $\varphi$ is identically zero on each $F_i$. And for any $x$ in $X$, for almost every $\omega$, all the cluster values of $(f_\omega^n(x))$ belong to a minimal set $F_i$ (Proposition \ref{adhe}), hence $\varphi(f_\omega^n(x))\to 0$, hence by integration over $\omega$, $P^n\varphi(x)\to 0$, and in particular, 
$$\varphi(x)=\lim_k P^{n_k}\varphi(x)=0.$$
Thus $\varphi$ is identically zero on $X$.
\end{proof}

\begin{proof}[Proof of Lemma \ref{indemin}]
If $F$ is a closed subset of $X$, let us set
$$\Theta(F)=\overline{\bigcup_{f\in \mathrm{supp}(\nu)} f(F)}.$$
We want to prove that if $F$ is a non empty closed subset such that $\Theta^p(F)\subset F$ then $F=X$. Set
$$\mathcal{F}=\{F\subset X \mbox{ closed },F\not=\emptyset,  \Theta^p(F)\subset F\},$$
and let $F$ be an element of $\mathcal{F}$ which is minimal with respect to the inclusion. Then:
\begin{itemize}
\item for any integer $k$, $\Theta^k(F)\in \mathcal{F}$ (obvious);
\item $\Theta^p(F)=F$ by minimality of $F$, since  $\Theta^p(F)\in\mathcal{F}$ and $\Theta^p(F)\subset F$;
\item for any integer $k$, $\Theta^k(F)$ is minimal with respect to the inclusion in $\mathcal{F}$: indeed, if $G\in \mathcal{F}$ and $G\subset \Theta^k(F)$ with $k<p$, then $\Theta^{p-k}(G)\in \mathcal{F}$ and  $\Theta^{p-k}(G)\subset F$, hence $\Theta^{p-k}(G)=F$ by minimality, and hence $\Theta^k(F)=\Theta^p(G)\subset G$.
\end{itemize}
We conclude that the sequence $(\Theta^k(F))_k$ is periodic (of period less than $p$), with elements that are pairwise disjoint or equal (by minimality). Let $p'$ the period of the sequence. Then the finite sequence  $F,\Theta(F)\ldots, \Theta^{p'-1}(F)$ is a sequence of pairwise disjoint closed sets such that any $f$ in $\mathrm{supp}(\nu)$ sends each set into the following, and the last one into the first. Because of the assumption of aperiodicity, $p'$ is necessarily equal to $1$. As a consequence, $\Theta(F)\subset F$, which means that $F$ is invariant by any $f$ in $\mathrm{supp}(\nu)$, and hence $F=X$ by minimality of the random walk.

\end{proof}

\subsection{Global contractions}

The following theorem shows that from the local phenomenon of contrations given by Assumption \ref{contracting}, we can obtain  a phenomenon of global contractions, in the sense that almost surely, the number of domains of attraction is finite: this result is close to a result of Y.Le Jan \cite{Lejan}.
\begin{prop}\label{LeJan}
We assume that the RDS satisfies assumption \ref{contracting}, and we assume moreover that $X$ is locally connected. Then there exists a positive integer $p$, such that, for $\mathbb{P}$-almost every $\omega$ in $\Omega$, there exists $p$  connected open sets $U_1(\omega),\ldots,U_p(\omega)$, pairwise disjoints, such that:
\begin{itemize}
\item the union $U(\omega)=U_1(\omega)\cup\cdots\cup U_p(\omega)$ is dense in $X$,
\item for every $i$ in $\{1,\ldots,p\}$, for every $x,y$ in $U_i(\omega)$,
$$d(f_\omega^n(x),f_\omega^n(y))\rightarrow 0.$$
\end{itemize}
\end{prop}
\begin{proof}
Let us consider the set 
$$\mathcal{E}=\{(\omega,x)\in \Omega\times S^1|(f_\omega^n)\mbox{ contracts a neighbourhood of } x\}=\bigcup_{\omega\in\Omega}\{\omega\}\times U(\omega).$$
 By Proposition \ref{finite}, there is a finite number of stationary probability measures $\mu_1,\ldots,\mu_d$. For each $i$ in $\{1,\ldots,d\}$, let $U_i(\omega)$ be the union of the connected components of $U(\omega)$ which have a positive $\mu_i$-measure. For $\mathbb{P}$-almost every $\omega$, the set $U_i(\omega)$ is an open subset with $\mu_i$-measure $1$, and has by Proposition \ref{randomcon}  a finite constant number $p_i$ of connected components. We write  $\tilde{U}(\omega)=U_1(\omega)\cup\cdots \cup U_d(\omega)$. As a consequence of Corollary \ref{randomset} applied to $\tilde{\mathcal{E}}=\bigcup_{\omega\in\Omega}\{\omega\}\times \tilde{U}(\omega)$, we know that  $\mathbb{P}\otimes \mu(\tilde{\mathcal{E}})=1$ for every probability measure $\mu$, and hence that $\tilde{U}(\omega)$ is dense for $\mathbb{P}$-almost every $\omega$. Thus, for $\mathbb{P}$-almost every $\omega$, $\tilde{U}(\omega)$ is a dense open subset of $X$ with a finite number $p=\sum_i p_i$ of connected components (and hence in fact, $U(\omega)=\tilde{U}(\omega)$). The result follows.
\end{proof}
We conclude with a criterion ensuring the synchronization of the RDS.
\begin{prop}
\label{globally contracting}
If the RDS satisfies Assumption \ref{contracting}, then the following assertions are equivalent:
\begin{enumerate}
\item  the random walk $\omega\mapsto (f_\omega^n)$ is synchronizing, i.e. for every $x,y$ in $X$, for $\mathbb{P}$-almost every $\omega$ in $\Omega$, $d(f_\omega^n(x),f_\omega^n(y))\xrightarrow[n\to+\infty]{}0$
\item  the random walk  $\omega\mapsto (f_\omega^n,f_\omega^n)$ admits a unique stationary probability measure on $X\times X$;
\item  The action of $G$ on $X$ is proximal, i.e. for every $x,y$ in $X$, there exists a sequence $(g_n)_n$ of elements of $G$ such that $d(g_n(x),g_n(y))\xrightarrow[n\to+\infty]{}0$.
\end{enumerate}
\end{prop}
\begin{proof} 
Let us notice that the random walk $\omega\mapsto (f_\omega^n,f_\omega^n)$ on $X\times X$ also satisfies the property of local contractions, so that the previous propositions of the section apply to it. We will denote by $\tilde{G}$ the semigroup associated to $\omega\mapsto (f_\omega^n,f_\omega^n)$, and by $D$ the diagonal of $X\times X$.

\medskip

$1 \Rightarrow 3$ is trivial.

\medskip

$3\Rightarrow 2$: By Proposition \ref{finite}, if there are two distinct ergodic stationary probability measures, then their respective topological supports $F_1$ and $F_2$ are two disjoint closed non empty subsets  of $X\times X$ invariant by $\tilde{G}$. Let $(x,y)$ be any point of $F_1$. By assumption, one can find a sequence of elements $g_n$ in $G$ such that the distance between $g_n(x)$ and $g_n(y)$ tends to $0$. Since $(g_n(x),g_n(y))\in F_1$, taking a cluster value of the sequence we deduce that $F_1$ intersects $D$ at some point $(z_1,z_1)$. In the same way, $F_2$ intersects $D$ at some point $(z_2,z_2)$. Choosing then a sequence $(h_n)$ in $G$ such that $d(h_n(z_1),h_n(z_2))\rightarrow 0$, any cluster value of $(h_n(z_1),h_n(z_1))$ is also a cluster value of $(h_n(z_2),h_n(z_2))$ and hence belongs to $F_1\cap F_2$, which is absurd.

\medskip

$2\Rightarrow 1$: By Proposition \ref{finite}, there is a unique minimal non empty closed subset $F$ invariant by $\tilde{G}$. Since $D$ is $\tilde{G}$-invariant, $F\subset D$. By Proposition \ref{adhe}, for every $(x,y)$ in $X\times X$, for $\mathbb{P}$-almost every $\omega$ in $\Omega$, the set of cluster values of $((f_\omega^n(x),f_\omega^n(y))_{n\in\N}$ is exactly $F$. In particular, it is included in $D$, hence $d(f_\omega^n(x),f_\omega^n(y))\xrightarrow[n\to+\infty]{}0$.
\end{proof}

\section{Proof of the main results}\label{sec-proofs}
We are going to combine Theorem \ref{main} proved in Section \ref{sec-inv} and the results of Section \ref{sec-consequences} to deduce Theorem \ref{distribution}, \ref{law}, \ref{Lejan-Antonov}, \ref{synchronizing} and their corollaries.

\subsection{Behaviour of random walks on $\mathrm{Homeo}(S^1)$}
\begin{proof}[Proof of Theorem \ref{distribution}]~

If we are in the first case of Corollary \ref{alternative}, then the result is a direct consequence of Proposition \ref{finite} and \ref{adhe}. If not, then we are in the second  case since $G$ has no finite orbit. That means that $G$ is semiconjugated to a minimal semigroup of isometries, and it is classical in this case that the stationary probability measure is unique: assuming up to the semiconjugation that $G$ is a semigroup of isometries acting minimally, if $\mu_1$ and $\mu_2$ are two ergodic stationary probabilities, one can find Birkhoff's points of $\mu_1$ and $\mu_2$ arbitrarily close, and then the trajectories of these points remain close, so that $\mu_1$ and $\mu_2$ are themselves arbitrarily close, hence equal. Thus the statinary probability measure $\mu$ is unique, and the convergence of $\frac{1}{N}\sum_{n=0}^{N-1}\delta_{f_{\omega}^n(x)}$ toward $\mu$ is for exemple a consequence of Proposition \ref{cluster}.
\end{proof}

\begin{proof}[Proof of Theorem \ref{law}]~

With the notations of the statement, the distribution $\mu_n^x$ is given by $\int \varphi d\mu_n^x=P^n\varphi(x)$ where $P$ is the transfer operator of the random walk, so that Proposition \ref{transfert} implies that $(\mu_n^x)_n$ converges in law, uniformly in $x$, to the stationary probability measure $\mu^x=\sum_{i=1}^du_i(x)\mu_i$ (keeping the notations of Proposition \ref{transfert}).
\end{proof}

\begin{proof}[Proof of Theorem \ref{Lejan-Antonov}]~

As a consequence of Proposition \ref{LeJan}, for $\mathbb{P}$-almost $\omega$ in $\Omega$, the set $U(\omega)$ of the points having a neighbourhood contracted by $(f_\omega^n)_n$ is dense and has a finite constant number $d$ of connected components, so that $S^1-U(\omega)$ is finite of cardinal $d$. To obtain the exponential contractions, it is enough to copy the proof of Proposition \ref{LeJan} replacing $U(\omega)$ by the set $U'(\omega)$  of the points having a neighbourhood contracted exponentially fast by $(f_\omega^n)_n$.
\end{proof}

\subsection{Synchronization}
\begin{proof}[Proof of Theorem \ref{synchronizing}]~
The only non trivial implication is iii) $\Rightarrow$ i).	 
Let us assume that the action of $G$ is proximal, that is, for any points $x$, $y$ there exists a sequence $g_n$ of elements of $G$ such that $\mathrm{dist}(g_n(x),g_n(y))\to 0$.\\

Firstly, let us justify that we are in the first case of Corollary \ref{alternative}:
 
If $G$ is semi-conjugated to $\tilde{G}$, then $\tilde{G}$ satisfies the same property of synchronization as $G$, so that $\tilde{G}$ is  not a group of isometries, and so we are not in second case.

If $G$ leaves invariant a finite set having at least two points, then the action of $G$ on this finite set cannot be proximal, which contradicts the assumptions. And $G$ cannot fix a singleton by assumption. Hence we are not in third case.\\

So we are in the first case, that is, the random walk satisfies the property of contractions given by Theorem \ref{main}. For any $x$, $y$ in the circle, one can find a sequence $g_n$ in $G$ such that $(g_n(x))_n$ and $(g_n(y))_n$ tend to a same point $c$. By Theorem \ref{main}, one can find a neighborhood of $c$ having positive probability to be contracted, hence we deduce that there is a set of $\omega$ with positive probability such that $\mathrm{dist}(f_\omega^n(x),f_\omega^n(y))$ tends to $0$ exponentially fast as $n$ tends to $+\infty$.\\

Let $\mathcal{E}$ be the set of $(\omega,x,y)$ in $\Omega\times S^1\times S^1$ such that $\mathrm{dist}(f_\omega^n(x),f_\omega^n(y))$ tends to $0$ exponentially fast as $n$ tends to $+\infty$. We obtained that for any $x$, $y$ in $S^1$, $\mathbb{P}\otimes \delta_{(x,y)}(\mathcal{E})>0$, hence by Proposition \ref{randomset}, we have actually $\mathbb{P}\otimes \delta_{(x,y)}(\mathcal{E})=1$, which means that the random walk is exponentially synchronizing.

\end{proof}
\begin{rem}
In the previous proof we could use Proposition \ref{globally contracting} to deduce the property of synchronization. However, it does not give the exponential speed.
\end{rem}

%
%
%
%

\begin{proof}[Proof of Corollary \ref{robust}]
Let $K$ be the compact minimal invariant by $G$ (necessarily unique because of the synchronization property).

\begin{lem}
There exists $g$ in $G$ having a robust fixed point and such that $g|_K\not=Id_K$.\\
(We say that $g$ has a robust fixed point if every small $C^0$-perturbation of $g$ has a fixed point)
\end{lem}
\begin{proof}
Let $x$ be any point of $K$. By Theorems \ref{main} and \ref{distribution}, one can find $\omega\in \Omega$ and a neighbourhood $I_0$ of $x$ such that $\mathrm{diam}(f_\omega^n(I_0))\to 0$ as $n\to +\infty$ and $(f_\omega^n(x))_{n\in\N}$ is dense in $K$. Thus we can find some integer $n$ such that $\overline{f_\omega^n(I_0)}\subset I_0-\{x\}$. Then $g=f_\omega^n$ satisfy $\overline{g(I_0)}\subset I_0$ (which implies that $g$ has a robust fixed point ) and $g(x)\not=x$.
\end{proof}

Let $g$ be as in the lemma, and $I$ be an open interval intersecting $K$ such that $g$ has no fixed point on the closure of $I$. Let $x$ and $y$ be in $S^1$. For almost every $\omega$, the trajectories $(f_\omega^n(x))$ and $(f_\omega^n(y))$ are by assumption asymptotically identical, and are dense in $K$. We deduce that we can find $h$ in $G$ such that $h(x),h(y)\in I$. By compactness, one can find $h_1,\ldots,h_p$ in $G$ such that for any $x,y$ in $S^1$, $h_i(x),h_i(y)\in I$ for some $i$ in $\{1,\ldots,p\}$.

Now, let $\tilde{f}_1,\ldots,\tilde{f}_d$ be small $C^0$-perturbations of the generators $f_1,\ldots,f_d$, $\tilde{G}$ be the semigroup generated by these new generators, and $\tilde{g},\tilde{h}_1,\ldots,\tilde{h}_p \in \tilde{G}$ be corresponding perturbations of $g,h_1,\ldots,h_p$. If the perturbations are small enough, the properties 
\begin{itemize}
\item $\forall x\in I, \tilde{g}(x)\not= x$,
\item $\tilde{g}$ has a fixed point,
\item $\forall x,y\in S^1 \exists i\in \{1,\ldots,p\}|\tilde{h}_i(x),\tilde{h}_i(y)\in I$,
\end{itemize}
 are still satisfied. The two first properties imply that $(\tilde{g}^n)$ converges to a constant on $I$, and hence using the third one we deduce that for any $x,y$ in $S^1$, there exists $i$ such that $\mathrm{dist}(\tilde{g}^n\circ \tilde{h}_i(x),\tilde{g}^n\circ \tilde{h}_i(y))\to 0$ as $n\to +\infty$. Thus, we can use Theorem \ref{synchronizing} to conclude that every random walk which is non degenerated on $\tilde{G}$ is synchronizing.
\end{proof}

\subsection{Random dynamical systems on $[0,1]$}

\begin{proof}[Proof of Corollary \ref{example}]
	Indentifying $I=[0,1]$ with an arc of $S^1$, we can prolong arbitrarily any injective map of $I$ to a homeomorphism of $S^1$. Thus, the result is a consequence of Theorem \ref{synchronizing}, once we have proved that
	\begin{itemize}
		\item There is no point of $I$ fixed by every element of $G$;
		\item There exists a sequence $(g_n)$ is $G$ such that 
		$$\mathrm{diam}(g_n(I))\xrightarrow[n\to+\infty]{}0.$$
	\end{itemize}
	
	The first point is straightforward, since a point fixed by $G$ belongs to $\bigcap_{g\in G} g(I)=\emptyset$.\\
	
	Let us prove the second point. Let us denote, for $g$ in $G$, $[a(g),b(g)]=g([0,1])$, $a=\sup_{g\in G}a(g)$ and $b=\inf_{g\in G}b(g)$. If $a\leq b$, then $[a,b]\subset \bigcap_{g\in G}[a(g),b(g)]\subset \bigcap_{g\in G} g(I)$, which is a contradiction. Thus $a>b$, so that one can find $g$ and $h$ in $G$ such that $a(g)>b(h)$, which implies that $g(I)\cap h(I)=\emptyset$. Since $g^2=g\circ g$ is increasing and has no fixed point on $h(I)$, we deduce that the sequence $(g^{2n})_n$ converges on $h(I)$ to a constant. In consequence, the sequence $g_n=g^{2n}\circ h$ satisfies the second point. That concludes the proof.
\end{proof}
\begin{proof}[Proof of Corollary \ref{mainR}]~

Identifying the points $0$ and $1$ gives a circle so that we can apply results of Section \ref{sec-inv} on the random walk.

Let $x_0$ be any point of $(0,1)$. For $\omega$ in $\Omega$, let
$$\mu_{N,\omega}=\frac{1}{N}\sum_{n=0}^{N-1}\delta_{f_\omega^n(x_0)}$$
We want to prove that for almost every $\omega$ in $\Omega$, the sequence $(\mu_{N,\omega})_{N\in\N}$ has some weak adherence value which is not invariant by $G$, in order to use Proposition \ref{contractuel} and Corollary \ref{invarianceiid}. In this view, let us note that the probability measures invariant by $G$ are necessarily convex combinations of $\delta_0$ and $\delta_1$.\\

Let $\mu$ a stationary probability measure of $G$ on $(0,1)$, that we can suppose ergodic. Since $\mbox{supp}(\mu)$ is invariant by $G$, we deduce that the interval $[\inf(\mbox{supp}(\mu)),\mbox{supp}(\mu)]$ is also invariant by $G$, hence is equal to $[0,1]$ by assumption, so that $0$ and $1$ belong to $\mbox{supp}(\mu)$. In particular one can find a Birkhoff point $a$ of $\mu$ in $(0,x_0)$ and an other Birkhoff point $b$ in $(x_0,1)$.\\

Let $I$ a compact interval of $(0,1)$ such that $\mu(I)\geq \frac{3}{4}$. Then for almost every $\omega$ in $\Omega$, the sets $A_\omega=\{n\in\N |f_\omega^n(a)\not\in I\}$ and $B_\omega=\{n\in\N |f_\omega^n(b)\not\in I\}$ have density smaller than $\frac{1}{4}$. But obviously, since $a<x_0<b$ the set $C_\omega=\{n\in\N |f_\omega^n(x_0)\not\in I\}$ is contained in $A_\omega\cup B_\omega$, hence this last set has density smaller than $\frac{1}{2}$.\\

In consequence, for almost every $\omega$ in $\Omega$, an adherance value $\mu'$ of $(\mu_{N,\omega})_{N\in\N}$ satisfies $\mu'(I)\geq \frac{1}{2}$, hence $\mu'$ is not a convex combination of $\delta_0$ and $\delta_1$, hence is not invariant by $G$ and
$$\lambda_{con}(\omega,x_0)\leq \lambda_{con}(\mathbb{P}\times\mu')<0$$
by Proposition \ref{contractuel} and Corollary \ref{invarianceiid}.

\end{proof}

\begin{proof}[Proof of Corollary\ref{torchder}]~

 By Corollary \ref{mainR}, every point $x$ has a contractible neighbourhood $I$ (in the meaning given in Definition \ref{contracdef}), and then there exists exactly one stationary ergodic probability measure $\mu_x$ such that $\mu_x(I)>0$: at most one because of Lemma \ref{contractible}, and at least one because $\mu(I)$ is positive.
 
  Then, $x\mapsto \mu_x$ is constant on contractible intervals, hence is locally constant, hence constant. This constant is the only ergodic stationary probability measure supported on $(0,1)$, and necassarily is $\mu$.
\end{proof}

\bibliography{Bibliography}
\bibliographystyle{plain}

\end{document}